\documentclass{article}
\usepackage[hmarginratio=1:1]{geometry}
\usepackage{titlesec}
\usepackage{pdfpages}
\titleformat{\section}  
  {\fontsize{14}{16}\bfseries} 
  {\thesection} 
  {1em} 
  {} 
  [] 

\titleformat{\subsection}  
  {\fontsize{12}{14}\bfseries} 
  {\thesubsection} 
  {1em} 
  {} 
  [] 
\definecolor{CycleDRB}{HTML}{4477AA}
\definecolor{CycleHodge}{HTML}{EE7733}
\definecolor{CycleAlg}{HTML}{228833}
\definecolor{ComparisonArrow}{HTML}{7547A5}

\usepackage[utf8]{inputenc}
\usepackage[utf8]{inputenc}
\usepackage{tikz, tikz-cd}
\usepackage{url, hyperref}
\usepackage{amsmath, amsthm, amssymb, xcolor, enumitem}
\usepackage{mathtools}
\usepackage{graphicx}
\usepackage{setspace}
\usepackage{fancyref}
\usepackage{hyperref}
\hypersetup{colorlinks,linkcolor={blue},citecolor={blue},urlcolor={red}}
\usepackage{biblatex}
\usepackage{comment}
\usepackage{breqn}
\newtheorem{theorem}{Theorem}[section]
\usepackage{mathtools}
\usepackage{floatrow}
\usepackage{newunicodechar}

\theoremstyle{definition}
\newtheorem{definition}[theorem]{Definition}
\newtheorem{remark}[theorem]{Remark}

\newtheorem{notation}[theorem]{Notation}

\newtheorem{setup}[theorem]{Setup}
\newtheorem{keyremark}[theorem]{Key Remark}

\theoremstyle{plain}
\newtheorem{proposition}[theorem]{Proposition}
\newtheorem{corollary}[theorem]{Corollary}
\newtheorem{lemma}[theorem]{Lemma}

\newtheorem{conjecture/theorem}[theorem]{Conjecture/Theorem}

\newtheorem{strategy}[theorem]{Strategy}
\newtheorem{proposition/definition}[theorem]{Proposition/Definition}
\newtheorem{tentative definition}[theorem]{Tentative Definition}

\newenvironment{explanation of the claim}{%
  \proof}{\endproof}

\newcommand{\gdrbmath}{\mathrm{G}_{\mathrm{dRB}}}

\newcommand{\resgmmath}{\mathrm{Res}_{E/\mathbb{Q}}\mathbb{G}_{\mathrm{m}}}
\newcommand{\gmmath}{\mathbb{G}_{\mathrm{m}} }

\newcommand{\gdrbhmath}{\mathrm{G}_{\mathrm{dRB}}^{\mathrm{h}} }

\newcommand{\absgalois}{\mathrm{Gal}(\overline{\mathbb{Q}}/\mathbb{Q})}
\newcommand{\galoisL}{\mathrm{Gal}(L/\mathbb{Q})}

\newcommand{\qbar}{\overline{\mathbb{Q}}}

\newcommand{\liegdrbmath}{\mathrm{g}_{\mathrm{dRB}}}
\newcommand{\liegdrbssmath}{\mathrm{g}_{\mathrm{dRB}}^{\mathrm{ss}}}
\newcommand{\liegdrbhmath}{\mathrm{g}_{\mathrm{dRB}}^{\mathrm{h}}}

\newcommand{\betti}{\mathrm{H}^1(A,\mathbb{Q})}

\newcommand{\sigmabar}{\overline{\sigma}}

\title{De Rham-Betti Groups of Type IV Abelian Fourfolds}
\author{Zekun Ji}

\date{}
\begin{document}
\maketitle
\begin{abstract}
We determine the de Rham-Betti (dRB) groups of several classes of abelian varieties over $\qbar$. We prove that $\gdrbmath(A)=\mathrm{MT}(A)$ for every simple abelian fourfold of type IV. At present, much of what is known concerning the dRB structures of abelian varieties derives from the Analytic Subgroup Theorem \cite{wustholz1989algebraische} by Wüstholz, whose applications primarily control the linear relations among periods of $\mathrm{H}^1(A)$ and divisor-class information. Compared to the theory of Hodge structures, our understanding of dRB structures is limited. Hence we adopt an approach different from Moonen-Zarhin's
method \cite{mz-4-folds} of determining the Mumford-Tate groups of these abelian varieties. Depending on the endomorphism type of these abelian fourfolds, we use Galois-theoretic analysis, half-twist construction \cite{van2001half} by van Geemen and positivity constraints arising from polarizations, as appropriate, to exclude reductive subgroups properly contained in Mumford-Tate groups as candidates for the dRB groups. We also use results on periods due to Gross \cite{gross1978periods} and Chudnovsky
\cite{chudnovsky1980algebraic}. This article is an expansion of the second part of the author's
PhD thesis \cite{ji2026thesis}; see also
\cite{ji2025rhambettigroupstypeiv}.
\end{abstract}
\tableofcontents
\section{Introduction}
The theory of de Rham-Betti (dRB) structures and groups was introduced by André as a framework for studying periods of smooth projective varieties defined over $\qbar$; see \cite[Section 7.6]{andre2004introduction}. In the setting of abelian varieties and hyper-Kähler varieties defined over $\qbar$, it is further studied by Bost-Charles in \cite{bost2016some} and Kreutz-Shen-Vial in \cite{ksv2026rhambetticlassescoefficients}. The main goal of this article, is to determine the de Rham-Betti groups of simple abelian fourfolds of type IV in the Albert classification defined over $\qbar$. 

\begin{theorem}[= Theorem
\ref{mainthmcm}, Theorem \ref{deg4theorem}, Theorem
\ref{imaginaryquadraticthm}, and Corollary \ref{corprimedim}]\label{introthm}
Let $A$ be an abelian variety defined over $\qbar$. Suppose $A$ belongs to one of the following classes, then the de Rham-Betti group of $A$ equals the Mumford-Tate group of $A$.
\begin{enumerate}
    \item A simple CM abelian fourfold;
    \item\label{casecmdeg4} An abelian fourfold whose endomorphism algebra is a quartic CM-field;
    \item\label{caseimaginaryquadratic} A simple abelian fourfold whose
    endomorphism algebra is an imaginary quadratic field;
    \item\label{casecmpprime} A simple CM abelian variety of prime dimension.
\end{enumerate} 
\end{theorem}

The Mumford-Tate groups of the abelian varieties appearing in Theorem \ref{introthm} have been explicitly determined by Moonen and Zarhin in \cite{mz-4-folds} and Ribet in \cite{ribet1983hodge}.
\subsection{De Rham-Betti and Hodge Structures}
\label{introsection1}
Let $X$ be a smooth projective variety defined over $\qbar$. Its Betti and algebraic de Rham cohomology groups are related by Grothendieck's comparison isomorphism $\rho:\mathrm{H}^{n}(X,\mathbb{Q})\otimes_{\mathbb{Q}}\mathbb{C}\cong\mathrm{H}_{\mathrm{dR}}^{n}(X/\qbar)\otimes_{\qbar}\mathbb{C}$; see \cite{grothendieck1966rham}. The resulting triple $(\mathrm{H}_{\mathrm{dR}}^{n}(X/\qbar),\mathrm{H}^{n}(X,\mathbb{Q}),\rho)$, which we denote by $\mathrm{H}^{n}_{\mathrm{dRB}}(X,\mathbb{Q})$, is a de Rham-Betti (dRB) structure. Recall that a dRB structure is a triple $(V_{\mathrm{dR}},V_{\mathrm{B}},\rho)$ where $V_{\mathrm{dR}}$ is a finite-dimensional $\qbar$-vector space, $V_{\mathrm{B}}$ is a finite-dimensional $\mathbb{Q}$-vector space, and $\rho:V_{\mathrm{B}}\otimes_{\mathbb{Q}}\mathbb{C}\cong V_{\mathrm{dR}}\otimes_{\qbar}\mathbb{C}$ is a $\mathbb{C}$-linear isomorphism. See \cite[Section 7.6]{andre2004introduction} as well as \cite[Section 2]{bost2016some} and \cite[Section 2]{ksv2026rhambetticlassescoefficients} for a more detailed exposition.

Given a dRB structure $(V_{\mathrm{dR}},V_{\mathrm{B}},\rho)$, an element $\alpha\in V_{\mathrm{B}}$ is called a \textit{de Rham-Betti class} if $\rho(\alpha\otimes 1)\in V_{\mathrm{dR}}\subset V_{\mathrm{dR}}\otimes_{\qbar}\mathbb{C}$. The Grothendieck period conjecture (GPC) predicts that every dRB class in the Tannakian category $\langle\mathrm{H}^{n}_{\mathrm{dRB}}(X,\mathbb{Q})\rangle^{\otimes}$ is the image of a $\mathbb{Q}$-coefficient algebraic cycle on the suitable power of $X$ under the cycle class map; see \cite[Section 1.1.3]{bost2016some}. At present even basic questions about the location of dRB classes are open. For example, it is not known \textit{in general} whether $\mathrm{H}^{2}_{\mathrm{dRB}}(X,\mathbb{Q})\otimes\mathbb{Q}(2\pi \sqrt{-1})$ contains dRB classes other than those coming from algebraic cycles. Moreover, beyond degree one, it is not known in general whether
nonzero dRB classes can occur in odd-degree cohomology. Thus it is a major challenge to determine the de Rham-Betti classes in $\langle\mathrm{H}^{n}_{\mathrm{dRB}}(X,\mathbb{Q})\rangle^{\otimes}$.

\begin{samepage}
We now focus on the setting of abelian varieties.
Firstly, for an abelian variety $A$ defined over $\qbar$, the \textit{known} relations among de Rham-Betti classes, Hodge classes and algebraic classes are summarized in Figure~\ref{fig:classes-and-groups}. By results of Deligne and Andr\'e, Hodge classes and motivated classes coincide for abelian varieties, while motivated classes are dRB classes; see \cite{andre2004introduction}. Thus the motivated classes and Hodge classes occupy the same region. The Grothendieck period conjecture predicts that the outermost region coincides with the innermost one while the Hodge conjecture predicts that the middle region coincides with the innermost one. The central theme of this article is to show, for the abelian varieties specified in Theorem \ref{introthm}, that the outermost region collapses to the middle region.
\begin{figure}[H]
\centering
\begin{tikzpicture}[
    font=\footnotesize,
    every node/.style={align=center},
    ring/.style={line width=0.9pt},
    comparison/.style={
        ->,
        >=stealth,
        line width=0.95pt,
        draw=ComparisonArrow,
        line cap=round
    }
]
    \path[use as bounding box] (-3.9,-2.4) rectangle (3.9,2.4);

    \filldraw[
        ring,
        draw=CycleDRB,
        fill=CycleDRB!8,
        rounded corners=16pt
    ] (-3.4,-2.0) rectangle (3.4,2.0);

    \filldraw[
        ring,
        draw=CycleHodge,
        fill=CycleHodge!10,
        rounded corners=14pt
    ] (-2.55,-1.30) rectangle (2.55,1.30);

    \filldraw[
        ring,
        draw=CycleAlg,
        fill=CycleAlg!12,
        rounded corners=12pt
    ] (-1.45,-0.72) rectangle (1.45,0.30);

    \node[font=\scriptsize\bfseries] at (0,1.62)
        {de Rham--Betti classes};

    \node[font=\scriptsize\bfseries] at (0,0.88)
        {motivated classes\\[-1pt] $=$ Hodge classes};

    \node[font=\scriptsize\bfseries] at (0,-0.22)
        {algebraic cycle\\classes};

    \draw[comparison] (-3.02, 1.72) -- (-2.72, 1.42);
    \draw[comparison] ( 3.02, 1.72) -- ( 2.72, 1.42);
    \draw[comparison] (-3.02,-1.72) -- (-2.72,-1.42);
    \draw[comparison] ( 3.02,-1.72) -- ( 2.72,-1.42);
\end{tikzpicture}

\caption{The known relations among the relevant classes in
$\langle\mathrm{H}^{1}(A,\mathbb{Q})\rangle^{\otimes}$, based on our current knowledge. Arrows indicate the desired collapsing direction.}
\label{fig:classes-and-groups}
\end{figure}
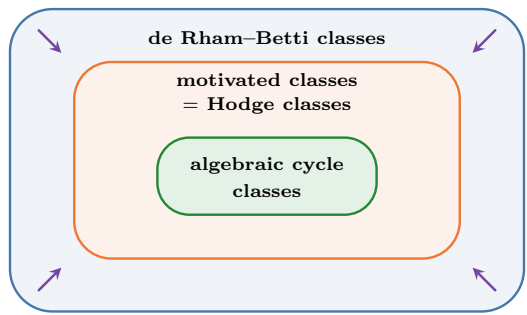
\end{samepage}

Moreover, for abelian varieties defined over $\qbar$, the deep ``Analytic Subgroup Theorem" of W\"ustholz \cite{wustholz1989algebraische} provides a key
transcendence-theoretic input. Using W\"ustholz's theorem, in
\cite[Section 7.5.3]{andre2004introduction}, Andr\'e remarks that
$\mathrm{H}^{1}_{\mathrm{dRB}}(A,\mathbb{Q})$ is a semisimple dRB
structure; see also \cite[Theorem 5.4]{ksv2026rhambetticlassescoefficients}. By the Tannakian formalism, the dRB classes in
$\langle\mathrm{H}^{1}_{\mathrm{dRB}}(A,\mathbb{Q})\rangle^{\otimes}$
are precisely the invariant tensors of the \textit{reductive} algebraic group
$$
\gdrbmath(A):=
\gdrbmath(\mathrm{H}^{1}_{\mathrm{dRB}}(A,\mathbb{Q}));
$$
For the detailed definition, see \cite[Section 2]{ksv2026rhambetticlassescoefficients} or
\cite[Section 2]{ji_1}. Recall that $\mathrm{MT}(A)$ is the algebraic subgroup of $\mathrm{GL}(\mathrm{H}^{1}(A,\mathbb{Q}))$ that fixes all the $(0,0)$-Hodge classes in
$\langle\mathrm{H}^{1}(A,\mathbb{Q})\rangle^{\otimes}$. By Tannakian duality, Figure~\ref{fig:classes-and-groups} then implies the following diagram in Figure~\ref{fig:tannakian-groups} relating the two groups. Hence the theme of this article can also be interpreted as showing that $\gdrbmath(A)$ is as large as possible, namely that it coincides with $\mathrm{MT}(A)$, for the abelian varieties specified in Theorem \ref{introthm}. 
\begin{samepage}

\begin{figure}[H]
\centering
\begin{tikzpicture}[
    font=\footnotesize,
    every node/.style={align=center},
    ring/.style={line width=0.9pt},
    comparison/.style={
        ->,
        >=stealth,
        line width=0.95pt,
        draw=ComparisonArrow,
        line cap=round
    }
]
    \path[use as bounding box]
        (-2.45,-2.45) rectangle (2.45,2.45);

    \filldraw[
        ring,
        draw=CycleHodge,
        fill=CycleHodge!10
    ] (0,0) circle (2.05cm);

    \filldraw[
        ring,
        draw=CycleDRB,
        fill=CycleDRB!10
    ] (0,0) circle (1.05cm);

    \node[font=\small]
        at (0,0) {$\gdrbmath(A)$};

    \node[font=\small]
        at (0,-1.55) {$\mathrm{MT}(A)$};

    \draw[comparison] (18:1.22)  -- (18:1.73);
    \draw[comparison] (145:1.22) -- (145:1.73);
    \draw[comparison] (302:1.22) -- (302:1.73);
\end{tikzpicture}

\caption{The inclusion of Tannakian groups corresponding to Figure~\ref{fig:classes-and-groups}, based on our
current knowledge. Arrows indicate the desired growing direction.}
\label{fig:tannakian-groups}
\end{figure}
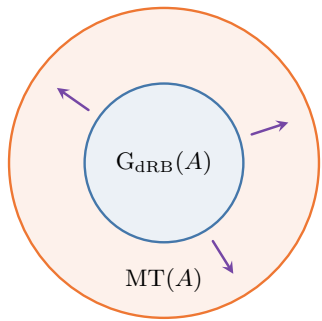
\end{samepage}

Another corollary of the Analytic Subgroup
Theorem concerns divisors. In \cite{bost2016some}, Bost and Charles prove the following result about dRB classes in degree two cohomology groups of abelian varieties.

\begin{theorem}[\cite{bost2016some}, Theorem 3.8]\label{bostoriginal}
Let $A$ be an abelian variety defined over $\qbar$. Every dRB class in $\mathrm{H}^2_{\mathrm{dRB}}(A,\mathbb{Q})\otimes \mathbb{Q}_{\mathrm{dRB}}(1)$ comes from a $\mathbb{Q}$-coefficient algebraic cycle on $A$. Moreover, we have that $\mathrm{End}(\mathrm{H}^1_{\mathrm{dRB}}(A,\mathbb{Q}))\cong\mathrm{End}^{\circ}(A)$.
\end{theorem}

\begin{remark}\label{astremark}
W\"ustholz's Analytic Subgroup Theorem is the principal
tool from transcendental number theory for studying de Rham-Betti structures of abelian varieties. Its applications in this setting, however, mainly concern $\mathrm{H}^1$ of an abelian variety or divisor classes (as in the results of Andr\'e and Bost-Charles recalled above), or more generally linear relations among periods of $1$-motives; see
\cite{huber2022transcendence}. Beyond such linear relations, very little is known about polynomial relations among periods, even for periods of CM abelian varieties; see the discussion in \cite[p.~2]{gao2026bi}. In the proof of Theorem \ref{introthm}, we encounter several scenarios where a direct period-theoretic argument would appear to require information about higher-degree polynomial relations among periods. See, for example Lemma \ref{p_0inpreimagelemma} and Section \ref{imaginaryquadraticsection}. We circumvent this difficulty by exploiting the geometry of the underlying abelian varieties; see for example Lemma~\ref{nosuchsimpleAV} and Proposition~\ref{antiweilabelianrealization}. More details about the outline of proof are given in Subsection \ref{proofoutline}.
\end{remark}

We now state a direct consequence of Theorem \ref{introthm}. In \cite[Table 1, p.~578]{mz-4-folds} and \cite[Theorem 0]{ribet1983hodge}, by performing invariant-theoretic computations, it was shown that the Hodge conjecture holds for all powers of certain abelian varieties listed in Theorem \ref{introthm}. We deduce:

\begin{corollary}
Let $A$ be either an abelian fourfold with quartic CM endomorphism field, or a simple CM abelian variety of prime dimension, or an anti-Weil type abelian fourfold defined over $\qbar$ (see Definition \ref{defnweiltype}). Then every dRB class in $\langle\mathrm{H}^{1}_{\mathrm{dRB}}(A,\mathbb{Q})\rangle^{\otimes}$ comes from a $\mathbb{Q}$-coefficient algebraic cycle on self-products of $A$.
\end{corollary}

We briefly recall what was known about de Rham-Betti groups of simple abelian varieties. For an arbitrary elliptic curve $E$ defined over $\qbar$, it is shown that $\gdrbmath(E)=\mathrm{MT}(E)$; see \cite[Theorem 5.13]{ksv2026rhambetticlassescoefficients} or \cite[Theorem 1]{kahn2025fullness}. It is also shown in \cite[Theorem 5.16]{ksv2026rhambetticlassescoefficients} that if $A$ is a simple abelian surface over $\qbar$ with $\mathrm{End}^{\circ}(A)\neq\mathbb{Q}$, then
$\gdrbmath(A)=\mathrm{MT}(A)$.

\subsection{Outline of the Proof}\label{proofoutline}
We now explain the strategy for proving Theorem \ref{introthm}. The group $\gdrbmath(A)$ contains the group of homotheties in
    $\mathrm{GL}(\betti)$; see \cite[Theorem 4.2]{ji_1}. Moreover, in
    \cite[Section 5]{ji_1}, a connected algebraic group
    $\gdrbhmath(A)\subset\mathrm{Hdg}(A)$ is constructed such that the
    following diagram commutes
$$\begin{tikzcd}
\gmmath \times \gdrbhmath(A) \arrow[r, two heads] \arrow[d, hook] & \gdrbmath(A) \arrow[d, hook] \\
\gmmath \times \mathrm{Hdg}(A) \arrow[r, two heads]               & \mathrm{MT}(A).              
\end{tikzcd}$$ Here the bottom arrow is the natural isogeny sending $(t,g)\in\gmmath\times\mathrm{Hdg}(A)$ to $t^{-1}g$. Therefore the problem is reduced to showing that $\gdrbhmath(A)=\mathrm{Hdg}(A)$.
\begin{strategy}\label{mainstrategy}
We analyze all the reductive $\mathbb{Q}$-algebraic subgroups properly contained in $\mathrm{Hdg}(A)$. For such a subgroup, if its invariants at degree two cohomology of $A$ do not agree with $(1,1)$-Hodge classes, then via Theorem \ref{bostoriginal}, we can rule out that subgroup as a candidate $\gdrbhmath(A)$. If a reductive subgroup passes the divisor test, we further exploit the geometry of the simple abelian variety $A$ in order to exclude this subgroup.
\end{strategy}
The proof of Theorem \ref{introthm} is organized into three cases: $A$ is a simple CM abelian variety; $A$ is an abelian fourfold with imaginary quadratic endomorphism field; $A$ is an abelian fourfold with quartic CM-endomorphism field. In each case, certain candidate reductive subgroups pass the divisor test and must therefore be excluded by distinctive arguments tailored to the geometry of $A$.

\refstepcounter{subsubsection}
\subsubsection*{\thesubsubsection\quad Simple CM Abelian Varieties} We start with the technical preparation for the case of simple CM abelian fourfolds in Section \ref{sectionprimecm}. For a simple CM abelian variety $A$ with endomorphism field $E$, we derive a useful criterion (Corollary \ref{divisortest}) for ruling out certain algebraic tori as candidates for $\gdrbhmath(A)\subset\mathrm{Hdg}(A)\subset\mathrm{U}_{E}$, in terms of their characters. Combined with the computation of periods of Weil structures by Gross in \cite{gross1978periods} and a theorem of Chudnovsky \cite{chudnovsky1980algebraic}, this leads to the equality $\gdrbmath(A)=\mathrm{MT}(A)$ for simple CM abelian varieties of prime dimension (see Corollary \ref{corprimedim}). This argument closely parallels Ribet's determination of the
Mumford-Tate groups of such abelian varieties
\cite[Theorem~2]{ribet1983hodge}, apart from the treatment of Weil
structures.

In Section \ref{cmgdrbsection}, we determine the de Rham-Betti groups of simple CM abelian fourfolds, see Theorem \ref{mainthmcm}. The Hodge groups of such abelian varieties were computed by Moonen and Zarhin in \cite{mz-4-folds}. Their method does not apply here. Crucially, Ribet's inequality (\cite[Section 3.5]{ribet1980division}) concerning the lower bounds for the dimension of Hodge groups is not known in the dRB setting; see Remark \ref{ribetinequalitu}. To get around this obstacle, in Subsection \ref{longcomp} and Subsection \ref{cmnonsimplicity}, we perform a thorough analysis of two-dimensional $\mathbb{Q}$-subtori of $\mathrm{U}_{E}$. For this, we will mainly use the divisor criterion from Corollary \ref{divisortest}.

In Case \ref{p0p2} and \ref{p0p3} of Lemma \ref{p_0inpreimagelemma}, certain two-dimensional $\mathbb Q$-tori pass the divisor test. Excluding these candidates for $\gdrbhmath(A)$ by a direct period-theoretic argument would require knowledge of higher-order polynomial relations among periods of the form specified in formula \eqref{p0p2eqs} and \eqref{p0p3eqs}. To get around this obstacle, we use the arithmetic properties of the endomorphism field $E$ and the fact that $A$ is a simple abelian variety to exclude these cases in Subsection \ref{cmnonsimplicity}. Finally we apply the argument from \cite[Section 7.6]{mz-4-folds} and the Weil structure result from Section \ref{sectionprimecm} to finish the proof.

\refstepcounter{subsubsection}
\subsubsection*{\thesubsubsection\quad Simple Abelian Fourfold with an Imaginary Quadratic Endomorphism Field}
\label{outline2}

In Section \ref{chaptercentre}, we consider simple type IV abelian fourfolds whose endomorphism field has degree two or four. For a Weil type abelian fourfold $A$, the same invariant-theoretic computation from \cite{mz-4-folds} can be applied to conclude that $\gdrbmath(A)=\mathrm{MT}(A)$.

The case of an anti-Weil type abelian fourfold $A$ (see Definition \ref{defnweiltype}) is more subtle: neither Ribet's method of determining $\mathrm{MT}(A)$ in \cite[Theorem 3]{ribet1983hodge} nor the straightforward invariant computation for divisors suffices. Indeed, the computations in Lemma \ref{wedgecandidateclassification} and Remark \ref{13remark} indicate that a direct period-theoretic argument would require knowledge of quadratic relations among periods of $A$, of which little is known; see discussion in \cite{gao2026bi}. Instead, in Proposition \ref{antiweilabelianrealization}, we use van Geemen's half-twist construction \cite{van2001half} to build a companion abelian sixfold $B$ over $\qbar$ associated with the anti-Weil type abelian fourfold $A$. This construction relates certain endomorphisms of $\mathrm{H}^2(A)$ to endomorphisms of $\mathrm{H}^1(B)$. We then apply Theorem \ref{bostoriginal} to $B$, which in turn excludes the existence of certain hypothetical de Rham-Betti endomorphisms of $\mathrm{H}^2_{\mathrm{dRB}}(A,\mathbb{Q})$ and thus exclude certain reductive subgroups of $\mathrm{Hdg}(A)$ as candidates for $\gdrbhmath(A)$. This proves Theorem \ref{imaginaryquadraticthm}. 
\refstepcounter{subsubsection}
\subsubsection*{\thesubsubsection\quad Simple Abelian Fourfold with a Quartic CM Endomorphism Field}

When the endomorphism algebra is a quartic CM-field, the method of determining the Mumford-Tate group in \cite{mz-4-folds} likewise does not carry over directly to the dRB setting. Instead, in the proof of Theorem~\ref{deg4theorem}, we exploit the positivity of the polarization form and the simplicity of $A$ to exclude the proper reductive subgroups of $\mathrm{Hdg}(A)$ as candidates for $\gdrbhmath(A)$. In more detail: if $\gdrbhmath(A)$ were one of the remaining proper $\mathbb{Q}$-algebraic reductive subgroups of $\mathrm{Hdg}(A)$---in this case, a $\qbar/\mathbb{Q}$-form of $\mathrm{SL}(2)$---then this assumption would force the values of any polarization form on $A$ to satisfy certain linear relations; see Lemma \ref{deg4Polarisation properties k merged}. We then show that no such bilinear form can polarize a weight-one Hodge structure associated with a simple abelian fourfold whose
endomorphism algebra is a quartic CM field.

\subsection{Acknowledgements}
This article expands on the second part of the author's PhD thesis written in University of Amsterdam (\cite{ji2026thesis}), supported by an NWO cluster grant with project number 613.009.153. I am very grateful to my advisor Mingmin Shen for introducing this topic to me and helpful discussions and suggestions. I have also benefited greatly from discussions with Charles Vial, Tobias Kreutz, Javier Fresán, Yichen Qin, and Mingyu Ni. Part of this work was carried out during a visit to the University of Science and Technology of China, and I would like to thank their hospitality.

\subsection{Notation and Conventions}
\begin{enumerate}
    \item We fix an algebraic closure $\qbar$ of $\mathbb{Q}$ and an embedding of fields $\qbar\xhookrightarrow{}\mathbb{C}$.
    \item Given an abelian variety $A$, we denote by $\mathrm{End}^{\circ}(A)$ its endomorphism algebra $\mathrm{End}(A)\otimes_{\mathbb{Z}}\mathbb{Q}$. 
    \item We fix a square root of $-1$ and denote it by $\sqrt{-1}$.
\end{enumerate}

\section{De Rham-Betti Groups of Simple CM Abelian Varieties}\label{sectionprimecm}
The main goal of this section is to set up notations and collect tools used in Section \ref{cmgdrbsection} and \ref{chaptercentre}. As a first application, we determine the dRB group of every simple CM abelian variety of prime dimension in Corollary \ref{corprimedim}. We first recall two results from \cite{ji_1}.
\begin{theorem}[\cite{ji_1}, Theorem 4.2 and Definition 5.1]\label{ji_1thmsummarize}
 Let $A$ be an abelian variety defined over $\qbar$. There is a connected and reductive $\mathbb{Q}$-algebraic group  $\gdrbhmath(A)\subset\mathrm{Hdg}(A)$, such that the following diagram commutes $$\begin{tikzcd}
\gmmath \times \gdrbhmath(A) \arrow[r, two heads] \arrow[d, hook] & \gdrbmath(A) \arrow[d, hook] \\
\gmmath \times \mathrm{Hdg}(A) \arrow[r, two heads]               & \mathrm{MT}(A)              
\end{tikzcd}.$$ The bottom arrow is the natural isogeny sending $(t,g)\in\gmmath\times\mathrm{Hdg}(A)$ to $t^{-1}g\in\mathrm{MT}(A)$.
\end{theorem}

\begin{lemma}[\cite{ji_1}, Corollary 5.5]\label{gdrbhinv}
Let $A$ be an abelian variety defined over $\qbar$. Denote $M:=\mathrm{H}^1_{\mathrm{dRB}}(A,\mathbb{Q})$ and let $m$ and $n$ be two non-negative integers such that $m-n$ is an even integer. Then
    $$(M^{\otimes m}\otimes M^{*\otimes n})^{\gdrbhmath(A)} \otimes \mathbb{Q}_{\mathrm{dRB}}({\frac{m-n}{2}}) = (M^{\otimes m}\otimes M^{*\otimes n}\otimes \mathbb{Q}_{\mathrm{dRB}}({\frac{m-n}{2}}))^{\gdrbmath(A)}.$$ 
\end{lemma}

\subsection{A Non-De Rham-Betti Group Criterion for Algebraic Tori}

For a $d$-dimensional simple abelian variety $A$ with complex multiplication by $E$, its Hodge group is a $\mathbb{Q}$-subtorus of $\mathrm{U}_{E}$. Thus so is $\gdrbhmath(A)$. We present a computational criterion which excludes certain (connected) subtorus of $\mathrm{U}_{E}$ as a candidate for $\gdrbhmath(A)$. We begin with the following set up of notations. 
\begin{setup}\label{torusinvsetup}
Denote $\betti$ by $V$. The morphism of $\mathbb{Q}$-algebras \begin{equation}\label{endoinclusion}E \xhookrightarrow{} \mathrm{End}(V)\end{equation} induces a homomorphism of $\mathbb{Q}$-algebraic groups $\mathrm{U}_E \xhookrightarrow{} \resgmmath\xhookrightarrow{}\mathrm{GL}(V)$. For any (connected) $\mathbb{Q}$-algebraic subtorus $T$ of $\mathrm{U}_E$, the basic correspondence between $\mathbb{Q}$-algebraic tori and their character groups (see \cite[Theorem 14.17]{milne2014algebraic}) yields a short exact sequence of $\absgalois$-modules 
\begin{equation}
\begin{tikzcd}\label{charseq}
0 \arrow[r] & \Delta \arrow[r, "\iota"] & X^{*}(\mathrm{U}_{E}) \arrow[r] & X^{*}(T) \arrow[r] & 0.
\end{tikzcd}
\end{equation} where $\Delta$ is the kernel of the natural surjective homomorphism $X^{*}(\mathrm{U}_{E})\twoheadrightarrow X^{*}(T)$.
The position of $\iota(\Delta)$ inside $X^{*}(\mathrm{U}_{E})$ records invariant tensors for the action of $T$; see Lemma \ref{torusinvariants}. Formula \eqref{endoinclusion} induces an eigenspace decomposition $$V_{\qbar}:=V\otimes_{\mathbb{Q}}\qbar=\bigoplus_{\sigma_{i}\in\mathrm{Hom}(E,\qbar)}V_{\sigma_{i}}$$ where each $V_{\sigma}$ is a one-dimensional $\qbar$-vector space and an element $e\in E$ acts on $V_{\sigma}$ via scalar multiplication by $\sigma(e)$. Choose a $\qbar$-basis $v_{\sigma}$ for each eigenspace $V_{\sigma}$. 
Let $S=\{\sigma_1,\ldots,\sigma_d\}$ be the CM-type of $A$, so that
\[
\mathrm{Hom}(E,\qbar)
=
S\coprod\overline{S},
\qquad
\overline{S}
=
\{\sigma_{d+1}:=\overline{\sigma_1},\ldots,
  \sigma_{2d}:=\overline{\sigma_d}\}.
\] Label the corresponding basis in $$X^{*}(\resgmmath)=\mathrm{Hom}(\resgmmath(\qbar),\gmmath(\qbar))=\mathrm{Hom}(\textstyle\prod_{\sigma_{i}\in\mathrm{Hom}(E,\qbar)}\qbar^{\times},\qbar^{\times})\cong \mathbb{Z}^{\mathrm{Hom}(E,\overline{\mathbb{Q}})}$$ by $\{e_{\sigma_1},\dots,e_{\sigma_{d}}\}$ and $\{e_{\sigma_{d+1}},\dots,e_{\sigma_{2d}}\}$ respectively. The kernel of the canonical surjective map $X^{*}(\resgmmath)\twoheadrightarrow X^{*}(\mathrm{U}_{E})$ is then spanned by $\{e_{\sigma_{i}}+e_{\sigma_{i+d}}|i\in\{1,\dots,d\}\}$, thus the image of $\{e_{\sigma_1},\dots,e_{\sigma_{d}}\}$, which we denote by $\{f_{\sigma_1},\dots,f_{\sigma_{d}}\}$, then forms a $\mathbb{Z}$-basis of $X^{*}(\mathrm{U}_{E})$. 
\end{setup}
\begin{lemma}\label{torusinvariants}
Keeping the same notation, suppose $\alpha:=\displaystyle\sum_{1\leq i\leq d} n_{i}f_{\sigma_i} \in \iota(\Delta),\, n_{i} \in \mathbb{Z}$. Then $\displaystyle v_{\alpha}:=\bigotimes_{1\leq i\leq d}v_{\sigma_i}^{\otimes n_i}$ is fixed by the induced action of $T({\qbar})$ on $\displaystyle V':=\bigotimes_{1\leq i\leq d}V_{\qbar}^{\otimes n_i}$. We adopt the convention that wherever $n<0$, $V^{\otimes n}:=V^{*\otimes (-n)}$ and $v_{\sigma_{i}}^{\otimes n}:=v_{\sigma_{i}}^{*\otimes (-n)}$.
\end{lemma}

\begin{proof}
An element $$(z_{\sigma_1},\dots,z_{\sigma_d},z_{\sigma_1}^{-1},\dots,z_{\sigma_d}^{-1})\in\mathrm{U}_{E}(\qbar)\subset\resgmmath(\qbar)=\prod_{\sigma_{i}\in\mathrm{Hom}(E,\qbar)}\qbar^{\times}$$ acts on $v_{\alpha}$ via scalar multiplication by $\displaystyle\prod_{1\leq i\leq d}z_{\sigma_{i}}^{n_{i}}$. 
Hence $v_{\alpha}$ is indeed an eigenvector of the induced action of $\mathrm{U}_{E}(\qbar)$ on $V'$ with eigencharacter $\alpha$. Since the representation of $T({\qbar})$ on $V'$ factors through $T({\qbar}) \xhookrightarrow{} \mathrm{U}_{E}(\qbar)$, $v_{\alpha}$ is also an eigenvector of $\alpha|_{T({\qbar})}$. Because $\alpha$ lies in the kernel of the map $X^{*}(\mathrm{U}_{E}) \twoheadrightarrow X^{*}(T)$, for any $t \in T({\qbar})$, we have $\alpha(t)=1$. Hence $T(\qbar)$ fixes $v_{\alpha}$.
\end{proof}

\begin{corollary} \label{divisortest}
With the same notations as Setup \ref{torusinvsetup}, if $\iota(\Delta)$ contains an element
$\pm f_{\sigma_i}\pm f_{\sigma_j}$ with $i\neq j$, then $T$ cannot be $\gdrbhmath(A)$.
\end{corollary}

\begin{proof}
Assume $T =\gdrbhmath(A)$. First suppose $\iota(\Delta)$ contains $f_{\sigma_{i}}-f_{\sigma_{j}}$ with $i\neq j$. Then by Lemma \ref{torusinvariants}, we would have $v_{\sigma_{i}} \otimes v^{-1}_{\sigma_{j}}\in(V_{\qbar} \otimes V_{\qbar}^{*})^{T(\qbar)} = \mathrm{End}(V_{\qbar})^{T(\qbar)} = \mathrm{End}_{\gdrbhmath(A)}(V) \otimes \qbar =\mathrm{End}_{\gdrbmath(A)}(V) \otimes \qbar = \mathrm{End}_{\mathrm{MT}(A)}(V) \otimes \qbar$, where the third equality follows from Lemma \ref{gdrbhinv} and the last equality follows from Theorem \ref{bostoriginal}. Note $\mathrm{End}_{\mathrm{MT}(A)}(V)=E$ and thus $\mathrm{End}_{\mathrm{MT}(A)}(V) \otimes \qbar$ is precisely spanned by $\{v_{\sigma_{i}} \otimes v^{-1}_{\sigma_{i}}|i \in \{1,\dots,2d\}\}$. Hence $v_{\sigma_{i}} \otimes v^{-1}_{\sigma_{j}}(i\neq j)$ cannot lie in it, a contradiction.

Now suppose $f_{\sigma_{i}}+f_{\sigma_{j}}\in \iota(\Delta)$ with $i\neq j$. Then by Lemma \ref{torusinvariants}, the element $v_{\sigma_{i}} \wedge v_{\sigma_{j}}$ would lie in $(\wedge^2V_{\qbar})^{T(\qbar)}$. Moreover, Lemma \ref{gdrbhinv} and Theorem \ref{bostoriginal} give that $(\wedge^2 V)^{\gdrbhmath(A)} \otimes_{\mathbb{Q}}\qbar=(\wedge^2 V)^{\mathrm{Hdg}(A)} \otimes_{\mathbb{Q}} \qbar$. Since the Picard rank of the simple CM abelian variety $A$ is $d$ and the CM-type $S$ of $A$ is precisely $\{\sigma_1,\dots,\sigma_d\}$, it follows that $(\wedge^2 V)^{\mathrm{Hdg}(A)}\otimes_{\mathbb{Q}}\qbar$ is precisely the $\qbar$-span of $\{v_{\sigma_{i}} \wedge v_{\sigma_{i+d}}|1 \leq i \leq d\}$.
Hence for $i,j\in \{1,\dots,d\}$ where $i\neq j$, the element $v_{\sigma_{i}} \wedge v_{\sigma_{j}}$ does not lie in $(\wedge^2 V)^{\mathrm{Hdg}(A)} \otimes_{\mathbb{Q}}\qbar$. Again this poses a contradiction.
\end{proof}
\subsection{Subtori of \texorpdfstring{$\mathrm{U}_{E}$}{UE}}
In this section we perform a preliminary study of subtori of \texorpdfstring{$\mathrm{U}_{E}$}{UE} and also set up the relevant notations.
\begin{setup}\label{setupcontinued}
Denote the Galois closure of $E/\mathbb{Q}$ in $\qbar$ by $L$. Recall $X^{*}(\resgmmath)$ admits a $\galoisL$-module structure; see \cite[Chapter 14]{milne2014algebraic}. In Setup \ref{torusinvsetup} we fix a basis of $X^{*}(\resgmmath)$ labeled by elements in $\mathrm{Hom}(E,L)$ and $\galoisL$ acts on $\mathrm{Hom}(E,L)$ by permutation. Thus the image of $\galoisL$ in $\mathrm{Aut}_{\mathbb{Z}}(X^{*}(\resgmmath))$ consists of permutation matrices. Denote this image by $G'$. By \cite[Proposition 1.1.4]{milne2006complex}, the complex conjugation in $\galoisL$ commutes with every other element of $\galoisL$. Because the kernel of the surjection from $X^{*}(\resgmmath)$ to $X^{*}(\mathrm{U}_{E})$ is spanned by $\{e_{\sigma_{i}}+e_{\sigma_{i+d}}|i\in\{1,\dots,d\}\}$, we have a well defined group homomorphism from $G'$ to $\mathrm{Aut}_{\mathbb{Z}}(X^{*}(\mathrm{U}_{E}))$. Denote the image of $G'$ in $\mathrm{Aut}_{\mathbb{Z}}(X^{*}(\mathrm{U}_{E}))$ by $G''$. With respect to the basis $\{f_{\sigma_1},\dots,f_{\sigma_{d}}\}$ of $X^{*}(\mathrm{U}_{E})$ constructed in Setup \ref{torusinvsetup}, elements of $G''$ are signed permutation matrices where both 1 and -1 are allowed as non-zero entries. Also note that $G' \cong G''$. For each element of $G''$, flipping -1 to 1, we obtain a forgetful group homomorphism $F: G'' \rightarrow \mathrm{S}_d$, the symmetric group on $d$ elements. The above construction can be summarized as follows.
 \begin{lemma} \label{galois_permutation}
     We have the following composition of group homomorphisms:
\begin{equation}\label{g'g''seq}
\begin{tikzcd}
\galoisL \arrow[r] & G' \arrow[r, "\cong"'] \arrow[d, hook]       & G'' \arrow[r,"F"] \arrow[d, hook]                    & H \arrow[d, hook] \\
                   & \mathrm{Aut}_{\mathbb{Z}}(X^{*}(\resgmmath)) & \mathrm{Aut}_{\mathbb{Z}}(X^{*}(\mathrm{U}_{E})) & \mathrm{S}_d     
\end{tikzcd}   
\end{equation}    
and H is a transitive subgroup of $\mathrm{S}_d$.
\end{lemma}

 \begin{proof}
     We only need to verify that $H$ is a transitive subgroup. But this is because the action of $\galoisL$ on $\mathrm{Hom}(E,L)$ is already transitive. 
 \end{proof}

\begin{remark}
We use the convention that the Galois action on $X^{*}(\resgmmath)$ is on the left and write elements in $X^{*}(\resgmmath)\otimes_{\mathbb{Z}}\mathbb{Q}$ as column vectors.
\end{remark}
\end{setup}
\begin{corollary}\label{biggerthan-one}Let $A$ be a simple abelian variety with complex multiplication by $E$. Suppose its dimension is at least two. Then the algebraic group $\gdrbhmath(A)$ has dimension greater than one.  
\end{corollary}

\begin{proof}
Denote $\gdrbhmath
(A)$ by $T$. First assume $T$ is zero-dimensional. Then because $T$ is connected by construction, we have $\iota(\Delta)=X^{*}(\mathrm{U}_{E})$, which clearly contradicts Corollary \ref{divisortest}. Suppose now $T$ is one-dimensional. Then $X^{*}(T)$ is a rank one free $\mathbb{Z}$-module and thus any $g \in G:=\galoisL$ acts on it as scalar multiplication by 1 or -1. Note that $\phi: X^{*}(\mathrm{U}_{E}) \rightarrow X^{*}(T)$ is a $G$-equivariant surjective map, we therefore have $\phi(g \circ e)=g \circ \phi(e)=\pm\phi(e)$ for any $e\in X^{*}(\mathrm{U}_{E})$. Also by Lemma \ref{galois_permutation}, given any two elements $x, y$ from the basis $\{f_{\sigma_1},\dots,f_{\sigma_d}\}$ of $X^{*}(\mathrm{U}_{E})$ fixed in Setup \ref{torusinvsetup}, there exists $g \in G$ such that $g \circ x=y$.  Hence for any pair $x,y\in \{f_{\sigma_1},\dots,f_{\sigma_d}\}$ with $x\neq y$, either $x-y $ or $x+y$ lies in $\mathrm{Ker}(\phi)=\iota(\Delta)$. But both cases violate Corollary \ref{divisortest}.   
\end{proof}
Now suppose $d=p$ is a prime number. Because $H$ acts transitively on a set of $p$-elements, it follows that $p||H|$. By Sylow's theorem, there exists a permutation cycle $\alpha\in H$ where $\alpha$ is of order $p$. Choose a signed permutation matrix $M\in G''$ with $F(M)=\alpha$. 
\begin{lemma}\label{lemmaprimecmclassification}
Suppose $p>2$. Then any non-zero $M$-stable $\mathbb{Q}$-vector subspace of $X:=X^{*}(\mathrm{U}_E)_{\mathbb{Q}}$ has dimension 1 or $p-1$ or is the entire space $X$.
\end{lemma}
\begin{proof}
The idea of proof is essentially the same as \cite[Theorem 2]{ribet1983hodge}. We recall it for later usage. Since $F(M)$ is a $p$-cycle, after reordering the fixed basis of
$X^*(\mathrm U_E)$ we may write
\begin{equation}\label{Maction}
 M\circ f_i=\varepsilon_i f_{i+1}\quad(1\leq i<p),
\qquad
M\circ f_p=\varepsilon_p f_1;
\end{equation}
where $\varepsilon_i\in\{\pm1\}$ and $M^{p}=\pm I$. Thus $M$ is semisimple. Therefore the characteristic polynomial of the $p\times p$-matrix $M$ equals $x^{p}-1$ or $x^{p}+1$.
Let $W\subset X$ be a $\mathbb{Q}$-vector subspace stable under $M$. Then $W_{\qbar}$ is spanned by the eigenvectors. First assume
$W_{\qbar}$ contains an eigenvector $w$ whose
eigenvalue is not $1$ or $-1$, i.e., it is a root of the polynomial $x^{p-1}+\dots+1$ or $x^{p-1}-\dots+1$. Then because $W_{\overline{\mathbb{Q}}}$ is Galois-stable, the Galois orbit of $w$ also lies in $W_{\qbar}$. Since $p$ is prime number, this orbit has size $p-1$, thus 
$\dim_{\overline{\mathbb{Q}}} W_{\overline{\mathbb{Q}}}\ge p-1$.
Otherwise, $W_{\overline{\mathbb{Q}}}$ is spanned only by the eigenvector
corresponding to the rational eigenvalue $1$ or $-1$ and therefore can only be of dimension $1$.
\end{proof}

\begin{corollary}\label{corprimdim}
Let $A$ be a simple CM abelian variety whose dimension is a prime number $p$. Then we have $\mathrm{dim}(\gdrbhmath(A))\geq p-1$.
\end{corollary}

\begin{proof}
Swap $T=\gdrbhmath(A)$ into \eqref{charseq}. If $p=2$, the same argument as Corollary \ref{biggerthan-one} can be used to rule out zero-dimensional algebraic torus as $\gdrbhmath(A)$. Assume
$p>2$. The space
$\iota(\Delta)\otimes_{\mathbb{Z}}\mathbb{Q}$ is $G''$-stable. Lemma~\ref{lemmaprimecmclassification} and Corollary \ref{biggerthan-one} imply that $\mathrm{rank}(\iota(\Delta))$ is either $0$ or $1$. Thus $\dim(T)=p-\mathrm{rank}(\iota(\Delta))\geq p-1$.

\end{proof}

\subsection{Weil Structures}
In this section we recall a well-known linear-algebraic construction first made by Weil (for example see \cite{Weil1979AbelianVA}); in particular its de Rham-Betti realization.

Let $A$ be an abelian variety defined over $\qbar$ with $V:=\betti$ and let $K$ be a number field acting on $V$ through a homomorphism of $\mathbb{Q}$-algebras $i: K \rightarrow \mathrm{End}_{\mathrm{Hdg}}(V)$. Over $\qbar$ this action gives an eigenspace decomposition
\begin{equation}\label{eigendecomp}V\otimes_{\mathbb{Q}} \qbar=\bigoplus_{\sigma \in \mathrm{Hom}(K,\qbar)}V_{\sigma}\end{equation} All eigenspaces share the same dimension $n=\frac{\mathrm{dim}(V)}{\mathrm{deg}(K/\mathbb{Q})}$. We now define the following $\qbar$-vector subspace of $\wedge^n V_{\qbar}$
\begin{equation*}
    W:=\bigoplus_{\sigma \in \mathrm{Hom}(K,\qbar)}\wedge^n V_{\sigma} \xhookrightarrow{} \wedge^n V\otimes\qbar.
\end{equation*} The natural action of $\absgalois$ on $V_{\qbar}$ permutes the eigenspaces $V_{\sigma}$. Therefore, the $\qbar$-subspace $W$ is $\absgalois$-stable. Thus $W$ descends to a $\mathbb{Q}$-subspace $W'$ of $\wedge^n V$. In the sequel, we will denote $W'$ by $\wedge_{K}^n V$.

\begin{definition} \label{weilstructure}
 We call the descended subspace, denoted by $W'=\wedge_{K}^n V$, the \textit{Weil structure} associated with the $K$-action on $V$.
\end{definition}

If $K$ preserves the Hodge or dRB structure on $V$, then $\wedge_{K}^n V$ is a Hodge substructure or a dRB substructure of $\wedge^n V$. This is because each eigenspace $V_{\sigma}$ from the decomposition (\ref{eigendecomp}) is preserved by the Mumford-Tate group or the de Rham-Betti group. More explicitly, in the
Hodge setting, for $\sigma \in \mathrm{Hom}(K,\qbar)$, denote by $m_{\sigma}$ the dimension of the $\mathbb{C}$-vector space $V_{\sigma}^{1,0}:=V_{\sigma}\otimes_{\qbar}\mathbb{C}\cap V^{1,0}$, then
\begin{equation}\label{formulaweilhodgenum}
\wedge_{K}^n \mathrm{H}^1(A,\mathbb{Q}) \otimes_{\mathbb{Q}} {\mathbb{C}}=\bigoplus_{\sigma \in \mathrm{Hom}(K,\qbar)}\wedge^n(V_{\sigma}^{1,0} \oplus V_{\sigma}^{0,1})=\bigoplus_{\sigma \in \mathrm{Hom}(K,\qbar)}\wedge^{m_{\sigma}}V_{\sigma}^{1,0}\otimes\wedge^{n-m_{\sigma}}V_{\sigma}^{0,1}.
\end{equation}

In particular, the summand indexed by $\sigma$ has Hodge type $(m_{\sigma},n-m_{\sigma})$.

\begin{definition}
\label{multiplicitydefn}
With the same notation as the above paragraph, we call the integer $m_{\sigma}$ the \textit{multiplicity} associated with the embedding $\sigma\in\mathrm{Hom}(K,\qbar)$. 
\end{definition}
In the dRB setting, in the case where $K$ is an imaginary quadratic field, Gross explicitly computed the comparison isomorphism for
the corresponding Weil dRB structures. 

\begin{lemma}[\cite{gross1978periods}, Theorem 3]\label{gross}
Let $A$ be an abelian variety over $\qbar$ of dimension $n$ with an imaginary quadratic field $K\xhookrightarrow{} \mathrm{End}^{\circ}(A)$. Suppose the (multi)set of multiplicities associated with $\mathrm{Hom}(K,\qbar)$ is $\{m,n-m\}$. Then with respect to suitable $\qbar$-basis, the comparison matrix of the Weil dRB structure $\wedge_{K}^{n}\mathrm{H}^1_{\mathrm{dRB}}(A,\mathbb{Q})$ is \begin{equation}\label{grosscomparison}\begin{pmatrix} b_{K}^{m} (2\pi \sqrt{-1}/b_{K})^{n-m} & 0 \\0 & b_{K}^{
n-m} (2\pi \sqrt{-1}/b_{K})^{m} \end{pmatrix}\end{equation} where $b_{K}$ is the Chowla-Selberg constant associated with the imaginary quadratic field $K$.
\end{lemma}
We shall need the following famous theorem by Chudnovsky (\cite{chudnovsky1980algebraic}).
\begin{theorem}[\cite{chudnovsky1980algebraic}]\label{thmchud}
 Given any imaginary quadratic field $K$, the Chowla-Selberg constant $b_{K}$ is a transcendental number and it is \textit{algebraically} independent from $2\pi \sqrt{-1}$.    
\end{theorem}
The following lemma classifies codimension 1 subtori of $\mathrm{U}_{E}$.
\begin{lemma}[\cite{mz-4-folds}, Key Lemma 7.3]\label{keylemmamz}
Let $E$ be a CM-field. Suppose $T$ is a (connected) codimension 1 subtorus of $\mathrm{U}_{E}$. Then there exists an imaginary quadratic subfield $K\xhookrightarrow{}E$ such that $T=\mathrm{SU}_{E/K}$, where $\mathrm{SU}_{E/K}$ is an algebraic group whose $\mathbb{Q}$-points are $\{x\in E^{\times}|\mathrm{det}_{K}(E\xrightarrow{\cdot x} E)=1\}.$ 
\end{lemma}

\begin{corollary}\label{corprimedim}
For any simple CM abelian variety $A$ with endomorphism field $E$, whose dimension is a prime number $p$, we have that $\gdrbmath(A)=\mathrm{MT}(A)$. 
\end{corollary}
\begin{proof}
The case where $\dim(A)=2$ follows from
\cite[Theorem~5.16, Case~(d)]{ksv2026rhambetticlassescoefficients},
using the method of \cite[Section~2.2]{moonen1999hodge}. Now assume $p$ is an odd prime number. By \cite[Theorem 2]{ribet1983hodge} and Theorem \ref{ji_1thmsummarize}, it suffices to show that $\gdrbhmath(A)=\mathrm{U}_{E}$. Corollary \ref{corprimdim} leaves only the possibility that $\gdrbhmath(A)$ has dimension $p-1$. In that case, Lemma \ref{keylemmamz} implies that $\gdrbhmath(A)=\mathrm{SU}_{E/K}=\mathrm{U}_E\cap \mathrm{SL}_{K}(V)$ where $K$ is an imaginary quadratic subfield of $E$.  By construction, elements in $W':=\wedge^{p}_{K}\mathrm{H}^1_{\mathrm{dRB}}(A,\mathbb{Q})$ are then fixed by the induced action of $\gdrbhmath(A)$. Thus by Lemma \ref{gdrbhinv}, every element in the dRB structure $$W'\otimes W'\otimes\mathbb{Q}_{\mathrm{dRB}}(p)\subset \mathrm{H}^{p}_{\mathrm{dRB}}(A,\mathbb{Q})\otimes \mathrm{H}^{p}_{\mathrm{dRB}}(A,\mathbb{Q})\otimes\mathbb{Q}_{\mathrm{dRB}}(p)$$ would be a dRB class. Unwinding the definitions, Gross's comparison formula
\eqref{grosscomparison} would imply  $$(b_{K}^{m} (2\pi \sqrt{-1}/b_{K})^{p-m})^2(2\pi \sqrt{-1})^{-p}=b_{K}^{4m-2p}(2\pi \sqrt{-1})^{p-2m}\in\qbar$$ for some $m\in\mathbb{Z}_{\geq 0}$. Since $p$ is odd, this contradicts Theorem \ref{thmchud}. Hence $\gdrbhmath(A)=\mathrm{U}_E$.
\end{proof}

\section{Determination of De Rham-Betti Groups of Simple CM Abelian Fourfolds}\label{cmgdrbsection}
The main result of this section is that for a simple CM abelian fourfold, its de Rham-Betti group is equal to its Mumford-Tate group (see Theorem \ref{mainthmcm}). 
\begin{remark}\label{ribetinequalitu}
The standard method of computing the Mumford-Tate group of such an abelian variety, as in \cite{mz-4-folds} or \cite{ribet1983hodge}, does not directly apply to the de Rham-Betti setting. In the Hodge theory setting, Ribet's inequality states that the dimension of the Mumford-Tate group of a simple CM abelian variety of dimension $d$ is bigger than or equal $2+\mathrm{log}_2(d)$; see \cite[Section 3.5]{ribet1980division}. Its proof uses, in an essential way, the definition of Mumford-Tate group as the smallest algebraic group defined over $\mathbb{Q}$ whose set of $\mathbb{R}$-points contains the image of the Deligne torus. Presently we have limited knowledge of the transcendence of periods or location of dRB classes and we lack such an equality for the lower bound of dRB groups for simple CM abelian varieties. To overcome this obstacle, we apply a delicate group-theoretical analysis in this section.
\end{remark}

\subsection{Classification of Two-Dimensional Subtori of \texorpdfstring{$\mathrm{U}_{E}$}{UE}}\label{longcomp}
We retain the notation of Section \ref{sectionprimecm}. Recall $A$ is a simple CM abelian fourfold with endomorphism field $E$, a CM-field of degree 8, and $\gdrbhmath(A)\subset\mathrm{Hdg}(A)\subset\mathrm{U}_{E}$. In this section we will study two-dimensional subtori of $\mathrm{U}_{E}$ and rule out \textit{most of them} as possible candidates for $\gdrbhmath(A)$ (see Lemma \ref{nolength4cycle}, Proposition \ref{noklein4} and Lemma \ref{A_4notpossible}). The main tools are the Galois-theoretic analysis of the short exact sequence (\ref{charseq}) from Setup \ref{torusinvsetup} and the criterion of Corollary \ref{divisortest}. Throughout this section, we will assume that $\gdrbhmath(A)$ is a two-dimensional subtorus of $\mathrm{U}_{E}$ and thus $\iota(\Delta)$ is a rank two submodule of $X^{*}(\mathrm{U}_{E})$ stable under the action of $G''$. As explained in Proposition \ref{noklein4} and Remark \ref{keyremarktotakeaway}, under certain arithmetic conditions on the endomorphism field, Case \ref{p0p2} and Case \ref{p0p3} of Lemma \ref{p_0inpreimagelemma} manifest the existence of two-dimensional $\mathbb{Q}$-subtori which cannot be excluded using a purely period-theoretic argument: their space of invariants goes beyond the reach of a direct application of Analytic Subgroup Theorem. We will deal with these cases in Subsection \ref{cmnonsimplicity}.

Recall from Lemma \ref{galois_permutation} that $H$, the image of $G''$, is a transitive subgroup of $\mathrm{S}_4$. Below we have a classification of all transitive subgroups of $\mathrm{S}_4$.

\begin{lemma}[\cite{rotman2012introduction}, Exercise 3.51]\label{transitiveclassification}
The possible transitive subgroups of $\mathrm{S}_4$ are: $\mathrm{S}_4$; $\mathrm{A}_4$; $V_4:=\{\mathrm{id},\,(12)(34),\,(13)(24),\,(14)(23)\}$; $\langle g \rangle \cong\mathbb{Z}/4\mathbb{Z}$; $\langle g,a \rangle\cong\mathrm{D}_4$, where $a\in V_4-\{\mathrm{id}\}$ and $g$ is a permutation cycle of order 4.
\end{lemma}

\begin{corollary}\label{transitiveins4}
A transitive subgroup $H$ of $\mathrm{S}_4$ satisfies at least one of the following conditions.
\begin{enumerate}
    \item\label{ginH} $H$ contains $4$-cycle in $\mathrm{S}_4$.
    \item\label{klein4inH} $H$ is equal to the Klein four-group $V_4$. 
    \item\label{A4inH} $H$ is equal to $\mathrm{A}_4$.
\end{enumerate}   
\end{corollary}

Using $\iota(\Delta)$ is stable under the action of $F^{-1}(H)=G''$, we can analyze properties of elements in $\iota(\Delta)$ case by case based on the above classification of $H$. And Case (\ref{ginH}) is dealt with in Lemma \ref{nolength4cycle}, Case (\ref{klein4inH}) is discussed in Proposition \ref{noklein4} and Case (\ref{A4inH}) is dealt with in Lemma \ref{A_4notpossible}. 

The following lemma will be used frequently.
\begin{lemma}\label{integralelement}
    Given $\iota:\Delta \xhookrightarrow{} X$ an inclusion of finite rank free $\mathbb{Z}$-modules, suppose $X/\iota(\Delta)$ is torsion free. Then we have $\iota(\Delta)=\iota(\Delta)\otimes_{\mathbb{Z}}\mathbb{Q}\cap X$.
\end{lemma}

\begin{strategy}\label{strategylongcomp}
We study rank two submodules of $X:=X^{*}(\mathrm{U}_E)$ stable under the action of $G''$. Pick a semisimple element $P\in G''$, then the candidate for $\iota(\Delta)\otimes_{\mathbb{Z}}\qbar$ is spanned by eigenvectors of $P$. It must also
be $G''$-stable, $\absgalois$-stable, saturated in $X$ (because $\gdrbhmath(A)$ is connected). Moreover, the candidate for $\iota(\Delta)$ cannot contain any element forbidden by Corollary \ref{divisortest}.
\end{strategy}

We treat Case (\ref{ginH}) from Corollary \ref{transitiveins4}.
\begin{lemma} \label{nolength4cycle}
  With notations from Setup \ref{setupcontinued} and Lemma \ref{galois_permutation}, if $T=\gdrbhmath(A)$ is a dimension 2 (connected) subtorus of $\mathrm{U}_{E}$, then $H\subset\mathrm{S}_4$ does not contain a $4$-cycle.
\end{lemma}

\begin{proof}
By symmetry we may assume $g=(1432)\in H$. Since $-1\in G''$, it suffices to
consider cases where $F^{-1}(g)$ contains a signed permutation matrix which has exactly \textit{zero}, \textit{one} or \textit{two} entries equal to $-1$.
\begin{enumerate}
\item \label{computations case_1M_1} Suppose $F^{-1}(g)$ contains the ordinary permutation matrix $M_1$ corresponding to the cycle $(1432)$. We now follow Strategy \ref{strategylongcomp}. Diagonalizing $M_1$ it has four distinct eigenvalues with the following eigenbasis: $\lambda_1=-1,\,v_1=(-1,1,-1,1)^{\mathrm{t}};\: \lambda_2=i,\,v_2=(i,-1,-i,1)^{\mathrm{t}};\:\lambda_3=-i,\,v_3=(-i,-1,i,1)^{\mathrm{t}};\:\lambda_4=1,\,v_4=(1,1,1,1)^{\mathrm{t}}$. Note that the eigenvalues are distinct, hence $\iota(\Delta)_{\qbar}$ is spanned by two eigenvectors of different eigenvalues. Also $\iota(\Delta)_{\qbar}$ is stable under the action of $\mathrm{Gal}(\mathbb{Q}(i)/\mathbb{Q})$. Therefore $\iota(\Delta)_{\mathbb{Q}}$ equals one of the following two scenarios 
\begin{itemize}
    \item $\mathrm{span}_{\mathbb{Q}}(\{v_{1},v_{4}\})=\mathrm{span}_{\mathbb{Q}}(\{(0,1,0,1)^{\mathrm{t}},(1,0,1,0)^{\mathrm{t}}\})$
    \item $\mathrm{span}_{\mathbb{Q}}(\{ v_2+v_3,\frac{v_2-v_{3}}{2i}\})=\mathrm{span}_{\mathbb{Q}}(\{(0,-1,0,1)^{\mathrm{t}}, 
(1,0,-1,0)^{\mathrm{t}}\}).$
\end{itemize}
Recall that we have fixed an integral basis for $X^{*}(\mathrm{U}_{\mathrm{E}})$ in Setup \ref{torusinvsetup}. Hence by Lemma \ref{integralelement}, we have that either $(1,0,1,0)^{\mathrm{t}} \in \iota(\Delta)$ or $(0,-1,0,1)^{\mathrm{t}}\in \iota(\Delta)$. And both cases contradict Corollary \ref{divisortest}.

\item \label{computations for Case1M2}
Suppose $F^{-1}(g)$ contains a signed permutation matrix $M_2$ with exactly \textit{one} $-1$ among its entries. Then $M_{2}^4+I=0$. The minimal polynomial of $M_2$ is $x^4+1$, which implies that $M_2$ has 4 distinct eigenvalues, on which $\absgalois$ acts transitively. Hence for the same reasoning as Lemma \ref{lemmaprimecmclassification}, no two-dimensional $\mathbb{Q}$-vector space can be stable under $M_2$.
\item \label{computations for case1m3}
Suppose $F^{-1}(g)$ contains a signed permutation matrix $M_3$ with exactly \textit{two} $-1$ among its entries. Then one can check $TM_1T^{-1}=M_3$ for some signed permutation matrix $T$. Hence any $M_3$-stable rank two submodule gives rise to an $M_1$-stable submodule $T\circ \iota(\Delta)$. By the analysis of Case \ref{computations case_1M_1} from this lemma and because $T$ is a signed permutation matrix, this contradicts Corollary \ref{divisortest}.
\end{enumerate}

\end{proof}
We now treat Case (\ref{klein4inH}) from Corollary \ref{transitiveins4}, i.e., the subgroup $H$ of $\mathrm{S}_4$ is isomorphic to the Klein four-group. Then different from Lemma \ref{nolength4cycle}, we do not immediately have a conclusive result.
\begin{proposition} \label{noklein4}
We use notations from Setup \ref{setupcontinued} and Lemma \ref{galois_permutation}. Suppose $T=\gdrbhmath(A)$ is a two-dimensional (connected) algebraic torus and the image $H$ of $G''$ in $\mathrm{S}_4$ is isomorphic to the Klein four-group. Then $G''$ is necessarily isomorphic to the dihedral group $\mathrm{D}_4$.
\end{proposition}
We make the following preparations for the proof of Proposition \ref{noklein4}.
\begin{setup}\label{klein4setup}
Let
\[
H=V_4=\{\mathrm{id},(12)(34),(13)(24),(14)(23)\};
\qquad
 g_1=(12)(34),\quad g_2=(14)(23),\quad g_3=(13)(24).
\]
Since $g_1,g_2,g_3$ are symmetric in $H$, we may relabel them if necessary. Because the element $-1$ lies in $G'' \subset \mathrm{Aut}_{\mathbb{Z}}(X^{*}(\mathrm{U}_{E}))$, it suffices to consider the following scenarios, which are not mutually exclusive.

\begin{enumerate}[label=(\Alph*)]\label{bigcases}
    \item\label{bigcase1} One element in $F^{-1}(g_1) \subset G''$ has exactly \textit{one} $-1$ among its entries.
    \item\label{bigcase2} One element in $F^{-1}(g_1) \subset G''$ is an ordinary permutation matrix.
    \item\label{bigcase3} One element in $F^{-1}(g_1)\subset G''$ has exactly \textit{two} $-1$ among its entries.
\end{enumerate}
\end{setup}
We have the following lemma dealing with Scenario \ref{bigcase1} from Setup \ref{klein4setup}.
\begin{lemma}\label{bigcase1lemma}
 In the setting of Proposition \ref{noklein4} and Setup \ref{klein4setup}, if there is a signed permutation matrix in $F^{-1}(g_1) \subset G''$ with exactly \textit{one} $-1$ among its entries, then $\gdrbhmath(A)$ cannot be an (connected) algebraic torus of dimensional two.
\end{lemma}
\begin{proof}
We follow Strategy \ref{strategylongcomp} as well. Suppose we have that $$M:=\scalebox{0.8}{$\begin{pmatrix} 0 & -1 & 0 & 0\\1 & 0 & 0 & 0\\0 & 0 & 0 & 1\\0 & 0 & 1 & 0 \end{pmatrix}$} \in F^{-1}(g_1).$$ Diagonalizing the matrix $M$, it has four distinct eigenvalues $\lambda_1=i,\,\lambda_2=-i,\,\lambda_3=1,\,\lambda_4=-1$ and the corresponding eigenvectors are $v_1=(1,-i,0,0)^{\mathrm{t}},\,v_2=(1,i,0,0)^{\mathrm{t}},\,v_3=(0,0,1,1)^{\mathrm{t}},\,v_4=(0,0,-1,1)^{\mathrm{t}}$. Then by a similar analysis to Lemma \ref{nolength4cycle}, either $(1,1,0,0)^{\mathrm{t}}$ or $(0,0,1,1)^{\mathrm{t}}$ is contained in $\iota(\Delta)$ and both cases violate Corollary \ref{divisortest}. Note any other signed permutation matrix which contains exactly \textit{one} $-1$ and whose image under $F$ is equal to $g_1$ is conjugate by a permutation matrix to $M$. Therefore they can be excluded in a similar manner.     
\end{proof}
The following lemma deals with Scenario \ref{bigcase2} from Setup \ref{klein4setup}.
\begin{lemma}\label{p_0inpreimagelemma}
We keep the assumptions of Proposition \ref{noklein4} and Setup \ref{klein4setup}. Furthermore, if one of the elements in $F^{-1}(g_1) \subset G''$ is an ordinary permutation matrix and if $\gdrbhmath(A)$ is a (connected) two-dimensional subtorus of $\mathrm{U}_{E}$, then $G''$ is necessarily isomorphic to $\mathrm{D}_4:=\langle a,x |a^4=1;\,x^2=1;\,axa=x\rangle$.
\end{lemma}
\begin{proof}
By symmetry and Lemma \ref{bigcase1lemma}, it suffices to study a
rank-two submodule stable under the an ordinary permutation matrix $$P_0:=\scalebox{0.8}{$\begin{pmatrix} 0 & 1 & 0 & 0\\1 & 0 & 0 & 0\\0 & 0 & 0 & 1\\0 & 0 & 1 & 0 \end{pmatrix}$}\in F^{-1}(g_1)$$ and one of the following preimages of $g_2$:
$$
(P_1,P_2,P_3,P_4)=\left(\scalebox{0.8}{$
\begin{pmatrix} 0 & 0 & 0 & 1\\0 & 0 & 1 & 0\\0 & 1 & 0 & 0\\1 & 0 & 0 & 0 \end{pmatrix},\,
\begin{pmatrix} 0 & 0 & 0 & 1\\0 & 0 & -1 & 0\\0 & 1 & 0 & 0\\-1 & 0 & 0 & 0 \end{pmatrix},\,
\begin{pmatrix} 0 & 0 & 0 & 1\\0 & 0 & -1 & 0\\0 & -1 & 0 & 0\\1 & 0 & 0 & 0 \end{pmatrix},\,
\begin{pmatrix} 0 & 0 & 0 & -1\\0 & 0 & -1 & 0\\0 & 1 & 0 & 0\\1 & 0 & 0 & 0 \end{pmatrix}$}\right).$$ Diagonalizing $P_{0}$, viewed as a $\mathbb{Q}$-linear transformation on $X^{*}({\mathrm{U}_{E}})_{\mathbb{Q}}$, its eigenvalues and eigenvectors are: $\lambda_1=-1,\,v_1=(0,0,-1,1)^{\mathrm{t}};\:\lambda_2=-1,\,v_2=(-1,1,0,0)^{\mathrm{t}};\:\lambda_3=1,\,v_3=(0,0,1,1)^{\mathrm{t}};\:\lambda_4=1,\,v_4=(1,1,0,0)^{\mathrm{t}}$.  By Corollary \ref{divisortest}, the set of eigenvalues of $P_0|_{\iota(\Delta)_{\mathbb{Q}}}$ can only be $\{-1,1\}$, i.e.,  \begin{equation}\label{p0module}\iota(\Delta)_{\mathbb{Q}}=\mathrm{span}_{\mathbb{Q}}(\{x_1v_1+x_2v_2,y_1v_3+y_2v_4\})=\mathrm{span}_{\mathbb{Q}}(\{(-x_1,x_1,-x_2,x_2)^{\mathrm{t}},(y_1,y_1,y_2,y_2)^{\mathrm{t}}\})\end{equation}
for some $x_{i},y_{i} \in \mathbb{Z}$.
\begin{enumerate}[label=(\alph*)]
\item\label{itemtwoordinary} Suppose $P_0, P_1$ are contained in $G''$. Then any vector space of the form (\ref{p0module}) is also preserved by $P_1$. Hence there exist $\alpha, \beta, \alpha^{'}, \beta^{'} \in \mathbb{Q}$ such that  
\begin{equation*}
\begin{split}
&P_1\circ(-x_1,x_1,-x_2,x_2)^{\mathrm{t}}=\alpha(-x_1,x_1,-x_2,x_2)^{\mathrm{t}}+\beta(y_1,y_1,y_2,y_2)^{\mathrm{t}}\\
&P_1\circ(y_1,y_1,y_2,y_2)^{\mathrm{t}}=\alpha'(-x_1,x_1,-x_2,x_2)^{\mathrm{t}}+\beta'(y_1,y_1,y_2,y_2)^{\mathrm{t}}.
\end{split}
\end{equation*}

We deduce that $x_1=\pm x_2$ and $y_1=\pm y_2$. Then Lemma \ref{integralelement} implies $(0,1,0,1)^{\mathrm{t}}$ or $(0,1,-1,0)^{\mathrm{t}}$ or $(0,1,1,0)^{\mathrm{t}}$ or $(0,1,0,-1)^{\mathrm{t}}$ lies in $\iota(\Delta)$. This gives a contradiction to Corollary \ref{divisortest}. 
\item\label{p0p2} Suppose $P_0, P_2$ are contained in $G''$. Then any vector space of the form (\ref{p0module}) must also be preserved by $P_2$, i.e., \begin{equation*}
\begin{split}
&P_2\circ(-x_1,x_1,-x_2,x_2)^{\mathrm{t}}=\alpha(-x_1,x_1,-x_2,x_2)^{\mathrm{t}}+\beta(y_1,y_1,y_2,y_2)^{\mathrm{t}}\\
&P_2\circ(y_1,y_1,y_2,y_2)^{\mathrm{t}}=\alpha'(-x_1,x_1,-x_2,x_2)^{\mathrm{t}}+\beta'(y_1,y_1,y_2,y_2)^{\mathrm{t}}.
\end{split}
\end{equation*} This puts the following condition on coordinates.  
\begin{equation}\label{p2}
x_1y_1=x_2y_2
\end{equation}
Then rank two submodules spanned by \begin{equation}\label{p0p2eqs}\{(-x_1,x_1,-x_2,x_2)^{\mathrm{t}},(x_2,x_2,x_1,x_1)^{\mathrm{t}}\}\end{equation} where $x_1,x_2\in\mathbb{Z}$ are stable under the action of $P_0$ and $P_2$. Upon close inspections, for suitably chosen $x_1,x_2\in\mathbb{Z}$, such modules \textit{cannot} be excluded using Corollary \ref{divisortest}. However, if $G''$ is exactly generated by $P_0$ and $P_2$, which is isomorphic to the dihedral group $\mathrm{D}_4=\langle a,x |a^4=1;\,x^2=1;\,axa=x\rangle$, via the group isomorphism \begin{equation}\label{p0p2d4}x=P_{0};\: a=P_2,\end{equation} then this satisfies the condition stated in the lemma.
\item\label{p0p3} Similarly, suppose that $P_0, P_3$ are contained in $G''$, then the same computation imposes the following condition on vector spaces of the form (\ref{p0module}):
\begin{equation}\label{p3}
 x_1y_1=-x_2y_2.
\end{equation} Then the rank two submodules spanned by \begin{equation}\label{p0p3eqs}\{(-x_1,x_1,-x_2,x_2)^{\mathrm{t}},(x_2,x_2,-x_1,-x_1)^{\mathrm{t}}\}\end{equation} where $x_1,x_2\in\mathbb{Z}$ are stable under the action of $P_0$ and $P_3$. For suitably chosen $x_1,x_2\in\mathbb{Z}$, they cannot be excluded using Corollary \ref{divisortest} either. Nevertheless, if $G''$ is exactly generated by $P_0$ and $P_3$, it is isomorphic to $\mathrm{D}_4$ via the group isomorphism \begin{equation}\label{p0p3d4}
    a=P_3P_0;\: x=P_0.
\end{equation} 
\item Suppose $P_0, P_4$ are contained in $G''$. This puts the condition $x_1^2+x_2^2=0$. Because $x_1,x_2\in\mathbb{Q}$, this contradicts that $\iota(\Delta)$ is of rank two.  
    
\end{enumerate}
In Case \ref{p0p2} and \ref{p0p3} discussed above, if the preimage of $g_3$ under $F$ 
contains elements other than $\pm P_0P_2$ and $\pm P_0P_3$, then $G''$ is not isomorphic to $\mathrm{D}_4$. We now deal with this scenario. Lemma \ref{bigcase1lemma}
excludes matrices with exactly one or three $-1$ entries in $F^{-1}(g_3)$, and Case
\ref{itemtwoordinary} of this lemma excludes ordinary permutation matrices in $F^{-1}(g_3)$. Hence only the following candidates can lie in $F^{-1}(g_3)$:
\[
(Q_1,Q_2)=\scalebox{.8}{$
\left(
\begin{pmatrix}0&0&-1&0\\0&0&0&-1\\1&0&0&0\\0&1&0&0\end{pmatrix},
\begin{pmatrix}0&0&-1&0\\0&0&0& 1\\1&0&0&0\\0&-1&0&0\end{pmatrix}
\right)$},
\]
for Case \ref{p0p2}, and
\[
(Q'_1,Q'_2)=\scalebox{.8}{$
\left(
\begin{pmatrix}0&0&-1&0\\0&0&0&-1\\1&0&0&0\\0&1&0&0\end{pmatrix},
\begin{pmatrix}0&0&-1&0\\0&0&0& 1\\-1&0&0&0\\0&1&0&0\end{pmatrix}
\right)$},
\]
for Case \ref{p0p3}. The computations of these branch cases are summarized in Table \ref{p_0table2}.
\begin{table}[ht]
    \centering
    \begin{tabular}{|c|c|c|}
    \hline
    Elements in $G''$  & Properties of $\iota(\Delta)$ & Can we rule out such $T$ \\
    \hline
    $P_0, P_2, Q_1$ & Can only be of rank 0 & Yes, See (\ref{p_0p_2q_1}) below\\
    \hline
    $P_0, P_2, Q_2$ & Contradicts Corollary \ref{divisortest} & Yes, See (\ref{p_0p_2q_2}) below\\
    \hline
    $P_0, P_3, Q'_1$ & Can only be of rank 0 & Yes, See (\ref{p_0p_3q'_1}) below\\
    \hline  
    $P_0, P_3, Q'_2$ & Contradicts Corollary \ref{divisortest} & Yes, See (\ref{p_0p_3q'_2}) below\\
    \hline
    \end{tabular}
    \caption{Branch Cases of \ref{p0p2} and \ref{p0p3}}
    \label{p_0table2}
\end{table}

\begin{enumerate}
\item\label{p_0p_2q_1} Suppose $F^{-1}(g_3)$ contains $Q_{1}$. Then $Q_{1}$ has to preserve a $\mathbb{Q}$-vector space of the form (\ref{p0module}). This gives the linear condition
$x_1^2+x_2^2=0$. Since $x_1,x_2\in\mathbb{Q}$, this contradicts that $\iota(\Delta)$ is of rank two, which is a contradiction. 

\item\label{p_0p_2q_2} If $F^{-1}(g_3)$ contains $Q_{2}$. This imposes the condition $x_1y_1=-x_2y_2$ for the $\mathbb{Q}$-vector space of the form (\ref{p0module}). But Case \ref{p0p2} of this lemma already imposes the condition $x_1y_1=x_2y_2$ (see formula (\ref{p2})). Hence this implies that $x_1y_1=x_2y_2=0$. Since $\iota(\Delta)$ has rank two, we deduce either $x_1=y_2=0$ or $x_2=y_1=0$. Lemma \ref{integralelement} then gives
$(0,0,-1,1)^{\mathrm t}\in \iota(\Delta)$ or $(-1,1,0,0)^{\mathrm t}\in \iota(\Delta)$, which contradicts Corollary \ref{divisortest}.   
\item\label{p_0p_3q'_1} Suppose $F^{-1}(g_3)$ contains $Q'_{1}:=Q_{1}$. This is identical to Case \eqref{p_0p_2q_1} above.
\item\label{p_0p_3q'_2}If $F^{-1}(g_3)$ contains $Q'_{2}$. Then we deduce that  $x_1y_1=x_2y_2$. But Case \ref{p0p3} of this lemma already imposes the linear condition $x_1y_1=-x_2y_2$ (see formula (\ref{p3})). Then by the same reasoning as Case (\ref{p_0p_2q_2}) above, this gives a contradiction to Corollary \ref{divisortest}. 
\end{enumerate}

\end{proof}
The following lemma deals with Scenario \ref{bigcase3} from Setup \ref{klein4setup}.
\begin{lemma}\label{bigcase3lemma}
 We keep the assumptions from Proposition \ref{noklein4} and Setup \ref{klein4setup}. If every signed permutation matrix in $F^{-1}(g_1)$, $F^{-1}(g_2)$ and $F^{-1}(g_3)$ contains exactly \textit{two} $-1$ among their entries, then $\gdrbhmath(A)$ cannot be a two-dimensional (connected) algebraic torus.
\end{lemma}
\begin{remark}\label{bigcase3lemmapatch}
Together, Lemma \ref{bigcase1lemma} and Lemma \ref{p_0inpreimagelemma} reduce
Scenario \ref{bigcase3} from Setup \ref{klein4setup} to the case described in Lemma
\ref{bigcase3lemma}. Indeed, Lemma \ref{bigcase1lemma} excludes signed permutation matrices with exactly one or three $-1$ entries, while Lemma \ref{p_0inpreimagelemma} handles ordinary permutation matrices. By
the symmetry of $g_1,g_2,g_3$ in $H$, the same reductions apply to every
$F^{-1}(g_i)$.
\end{remark}

\begin{proof}[Proof of Lemma \ref{bigcase3lemma}]
Label the following potential elements in $F^{-1}(g_1)$ which contain \textit{two} $-1$ among their entries by $$(P_0',P_1',P_2')=\scalebox{0.8}{$\left(\begin{pmatrix} 0 & -1 & 0 & 0\\-1 & 0 & 0 & 0\\0 & 0 & 0 & 1\\0 & 0 & 1 & 0 \end{pmatrix},\, \begin{pmatrix} 0 & -1 & 0 & 0\\1 & 0 & 0 & 0\\0 & 0 & 0 & -1\\0 & 0 & 1 & 0 \end{pmatrix},\,\begin{pmatrix} 0 & -1 & 0 & 0\\1 & 0 & 0 & 0\\0 & 0 & 0 & 1\\0 & 0 & -1 & 0 \end{pmatrix}\right).$}$$ Also label the following potential elements in $F^{-1}(g_2)$ which contain \textit{two} $-1$ among their entries by  $$(Q_0',Q_1',Q_2')=\scalebox{0.8}{$\left(\begin{pmatrix} 0 & 0 & 0 & 1\\0 & 0 & -1 & 0\\0 & -1 & 0 & 0\\1 & 0 & 0 & 0 \end{pmatrix},\, \begin{pmatrix} 0 & 0 & 0 & 1\\0 & 0 & -1 & 0\\0 & 1 & 0 & 0\\-1 & 0 & 0 & 0 \end{pmatrix},\,\begin{pmatrix} 0 & 0 & 0 & -1\\0 & 0 & -1 & 0\\0 & 1 & 0 & 0\\1 & 0 & 0 & 0 \end{pmatrix}\right).$}$$ By the assumptions of the lemma, preimages of $g_3$ consist only of matrices with exactly two $-1$ among their entries and the products $P_0'Q_2', P_1'Q_1', P_2'Q_0'$ are ordinary permutation matrices over $g_3$. Therefore we only need to check the six cases summarized in Table \ref{tableeachtwo-1}.  

\begin{table}[ht]
    \centering
    \begin{tabular}{|c|c|c|}
    \hline
    Elements in $G''$  & Properties of $\iota(\Delta)$ & Can we rule out such $T$ \\
    \hline
    $P_{0}', Q_{0}'$ & Contradicts Corollary \ref{divisortest} & Yes, See (\ref{p'_0q'_0})\\
    \hline
    $P_{0}', Q_{1}'$ & Can only be of rank 0 & Yes, See (\ref{p'_0q'_1})\\
    \hline
    $P_{1}', Q_{0}'$ & Can only be of rank 0 & Yes, See (\ref{p'_1q'_0})\\
    \hline
    $P_1',Q_2'$ &Can only be of rank 0& Yes, see (\ref{p'_1q'_2})\\
    \hline
    $P_2',Q_2'$ &Contradicts Corollary \ref{divisortest} & Yes, see (\ref{p'_2q'_2})\\
    \hline
    $P_2',Q_1'$ &Can only be of rank 0& Yes, see (\ref{p'_2q'_1})\\
    \hline
    \end{tabular}
    \caption{Summary of Computations for Lemma \ref{bigcase3lemma}}
    \label{tableeachtwo-1}
\end{table}
\begin{enumerate}
\item\label{p'_0q'_0} Suppose $P_{0}'\in F^{-1}(g_1)$ and $Q_{0}'\in F^{-1}(g_2)$. Then $P_0'$ has eigenvalue $\lambda_1=1$ whose eigenspace is spanned by $(0,0,1,1)^{\mathrm{t}}$ and $(-1,1,0,0)^{\mathrm{t}}$ and eigenvalue $\lambda_2=-1$ whose eigenspace is spanned by $(0,0,-1,1)^{\mathrm{t}}$ and $(1,1,0,0)^{\mathrm{t}}$. If $\iota(\Delta)_{\mathbb{Q}}$ equals the eigenspace associated with $\lambda_1$ or $\lambda_2$, this contradicts Corollary \ref{divisortest}. Thus \begin{equation}\label{p_0'space}\iota(\Delta)_{\mathbb{Q}}=\mathrm{span}_{\mathbb{Q}}(\{(-x_1,x_1,x_2,x_2)^{\mathrm{t}},(y_1,y_1,-y_2,y_2)^{\mathrm{t}}\})\end{equation} for some $x_{i},y_{i} \in \mathbb{Z}$. The condition that  $\iota(\Delta)_{\mathbb{Q}}$ is stable under $Q_{0}'$ indicates that $x_1^2=x_2^2$ and $y_1^2=y_2^2$. Then the same argument for Lemma \ref{bigcase1lemma} can be applied to exclude this case.
\item\label{p'_0q'_1} Suppose $P_{0}'\in F^{-1}(g_1)$ and $Q_{1}'\in F^{-1}(g_2)$. Then by the previous Case (\ref{p'_0q'_0}), it suffices to consider subspaces of the form (\ref{p_0'space}) which are also stable under $Q_1'$. This implies $x_1^2+x_2^2=0$. Because $x_{i}\in\mathbb{Q}$, this holds only if $x_1=x_2=0$, which gives a contradiction. 
\item\label{p'_1q'_0}Suppose $P_{1}'\in F^{-1}(g_1)$ and $Q_{0}'\in F^{-1}(g_2)$. The permutation matrix $R$ corresponding to the cycle $(1432)$ satisfies $RP'_0R^{-1}=-Q'_0$ and $RQ'_1R^{-1}=P'_1$. Hence if $\iota(\Delta)_{\mathbb{Q}}$ were stable under $P_{1}'$ and $Q_{0}'$, then $R^{-1}(\iota(\Delta)_{\mathbb{Q}})$ would be stable under $P_{0}'$ and $Q_{1}'$, contradicting Case \eqref{p'_0q'_1} of this lemma.
\item\label{p'_1q'_2}  
Suppose $P_1'\in F^{-1}(g_1)$ and $Q_2'\in F^{-1}(g_2)$. Then $Q_2'$ has eigenvalue $\lambda_1=\sqrt{-1}$ with eigenbasis equal to $\{(\sqrt{-1},0,0,1)^{\mathrm{t}},(0,\sqrt{-1},1,0)^{\mathrm{t}}\}$ and eigenvalue $\lambda_2=-\sqrt{-1}$ with eigenbasis equal to $\{(-\sqrt{-1},0,0,1)^{\mathrm{t}},(0,-\sqrt{-1},1,0)^{\mathrm{t}}\}$. The two eigenspaces do not descend to $\mathbb{Q}$. Therefore it suffices to consider $\mathbb{Q}(\sqrt{-1})$-vector subspaces of the form \begin{equation}\label{spaceq2'}\mathrm{span}_{\mathbb{Q}(\sqrt{-1})}(\{(\sqrt{-1}x_1,\sqrt{-1}x_2,x_2,x_1)^{\mathrm{t}},(-\sqrt{-1}y_1,-\sqrt{-1}y_2,y_2,y_1)^{\mathrm{t}}\});x_1,x_2,y_1,y_2\in \mathbb{Q}(\sqrt{-1})\end{equation} as candidates for $\iota(\Delta)_{\mathbb{Q}(\sqrt{-1})}$. Because $\iota(\Delta)_{\mathbb{Q}(\sqrt{-1})}$ is stable under the action of $P_1'$, in this case we deduce that \begin{equation}\label{p1'q2'}x_1y_1=-x_2y_2.\end{equation} Since we cannot have $x_1=x_2=0$ or $y_1=y_2=0$, $\iota(\Delta)_{\mathbb{Q}(\sqrt{
-1
})}$ is spanned by $$\{(\sqrt{-1}x_1,\sqrt{-1}x_2,x_2,x_1)^{\mathrm{t}},(\sqrt{-1}x_2,-\sqrt{-1}x_1,x_1,-x_2)^{\mathrm{t}}\};x_1,x_2\in \mathbb{Q}(\sqrt{-1})$$ We claim such $\mathbb{Q}(\sqrt{-1})$-vector spaces cannot descend to $\mathbb{Q}$. Suppose it were defined over $\mathbb{Q}$, then it would be stable under the complex conjugation. Thus for some $m,n\in \mathbb{Q}(\sqrt{-1})$ 
\begin{equation*}
m(\sqrt{-1}x_1,\sqrt{-1}x_2,x_2,x_1)^{\mathrm{t}}+n(\sqrt{-1}x_2,-\sqrt{-1}x_1,x_1,-x_2)^{\mathrm{t}}=(-\sqrt{-1}\overline{x_1},-\sqrt{-1}\overline{x_2},\overline{x_2},\overline{x_1})^{\mathrm{t}}.
\end{equation*} We thus deduce that $nx_1=\overline{x_2}, -nx_2=\overline{x_1}$. Because $\iota(\Delta)_{\mathbb{Q}(\sqrt{-1})}$ is two-dimensional, this leads to $x_1\overline{x_1}+x_2\overline{x_2}=0$ and hence $x_1=x_2=0$, which poses a contradiction.
\item\label{p'_2q'_2} Suppose $\iota(\Delta)_{\mathbb{Q}}$ is stable under $P_2'\in F^{-1}(g_1)$ and $Q_2'\in F^{-1}(g_2)$. Note there exists $$S=\scalebox{0.8}{$\begin{pmatrix} 0 & 1 & 0 & 0\\0 & 0 & -\sqrt{-1} & 0\\0 & 0 & 0 & 1\\-\sqrt{-1} & 0 & 0 & 0 \end{pmatrix}$}\in\mathrm{GL}(X_{\mathbb{Q}(\sqrt{-1})})$$ such that $SP_2'S^{-1}=\sqrt{-1}Q_0'$ and $SQ_2'S^{-1}=\sqrt{-1}P_0'$. Hence suppose the two-dimensional $\mathbb{Q}$-vector space $\iota(\Delta)$ is stable under $P_2'$ and $Q_2'$, then $S\circ \iota(\Delta)$ is a $\mathbb{Q}(\sqrt{-1})$-vector space stable under $P_0'$ and $Q_0'$. By Case
\eqref{p'_0q'_0} of this lemma, it contains either $(0,1,0,1)^{\mathrm t}$ or
$(1,0,1,0)^{\mathrm t}$. Applying $S^{-1}$ and clearing the $\sqrt{-1}$ factor, this gives a contradiction to Corollary \ref{divisortest}.

\item\label{p'_2q'_1} Finally suppose $\iota(\Delta)_{\mathbb{Q}}$ is stable under $P_2'$ and $Q_1'$. Then the permutation matrix $T$ corresponding to the cycle $(1432)$ satisfies $TP'_2T^{-1}=-Q_2'$ and $TQ_1'T^{-1}=P_1'$. Thus any two-dimensional $\mathbb Q$-space stable under $P_2'$ and
$Q_1'$ would be carried by $T$ to one stable under $P_1'$ and $Q_2'$, which is impossible by Case \eqref{p'_1q'_2} of this lemma.
\end{enumerate}
All possible cases lead to contradictions. Hence $\gdrbhmath(A)$ cannot be a two-dimensional algebraic torus under the assumptions of the lemma.
\end{proof}

\begin{proof}[Proof of Proposition \ref{noklein4}]
 The three scenarios from Setup \ref{klein4setup} are handled by Lemma \ref{bigcase1lemma}, Lemma \ref{p_0inpreimagelemma}, Lemma \ref{bigcase3lemma} as well as Remark \ref{bigcase3lemmapatch}. The only
case not excluded is the one described in Lemma
\ref{p_0inpreimagelemma}, where $G''\cong \mathrm D_4$.  This proves
the proposition.
\end{proof}
We now treat Case (\ref{A4inH}) from Corollary \ref{transitiveins4} where $H$ is equal to $\mathrm{A}_4$.
\begin{lemma} \label{A_4notpossible}
We use notations from Setup \ref{setupcontinued} and Lemma \ref{galois_permutation}. If $T=\gdrbhmath(A)$ is a two-dimensional (connected) algebraic torus, then the image $H$ of $G''$ in $\mathrm{S}_4$ cannot be $\mathrm{A}_4$.  
\end{lemma}
\begin{proof}

Suppose $H=\mathrm{A}_4$. 
Then $H$ contains all 3-cycles in $\mathrm{S}_4$ and the Klein four-group. Choose an element $M\in F^{-1}((123))\subset G''$. In the
basis fixed in Setup \ref{torusinvsetup}, $M$ has the form $$\scalebox{0.8}{$\begin{pmatrix} 0 & 0 & * & 0\\ * & 0 & 0 & 0\\0 & * & 0 & 0\\0 & 0 & 0 & * \end{pmatrix}$}$$ with $*=\pm 1$. Note that $M$ has the eigenspace decomposition of the following form: $v_1=(0,0,0,1)^{\mathrm{t}}$ with eigenvalue $\pm 1$; $v_2=(a,b,c,0)^{\mathrm{t}}$ with eigenvalue $\pm 1$ where $a,b,c=\pm 1$; $v_3=(a',b',c',0)^{\mathrm{t}}$; $v_4=(a'',b'',c'',0)^{\mathrm{t}}$ with eigenvalues either equal to the roots of the irreducible polynomial $x^2-x+1$ or $x^2+x+1$. Hence the $G''$-stable two-dimensional vector space $\iota(\Delta)_{\mathbb{Q}}$ is either spanned by $v_1$ and $v_2$ or after extending scalars to $\qbar$, spanned by $v_3$ and $v_4$. However $H=\mathrm{A}_4$ contains $V_4$. If $\iota(\Delta)_{\mathbb{Q}}$ were spanned by $v_1$ and $v_2$, then for any $N'\in F^{-1}((14)(23))$, $N'\circ v_1$ is of the form $(\pm 1,0,0,0)$, which does not lie in $\mathrm{span}_{\mathbb{Q}}(\{v_1,v_2\})$. If $\iota(\Delta)_{\qbar}$ were spanned by $v_3$ and $v_4$, then $N'\circ v_3$ has a non-zero entry in the last coordinate, and clearly does not lie in $\iota(\Delta)_{\qbar}$.  Both possibilities lead to a contradiction.
\end{proof}
Therefore the three cases from Corollary \ref{transitiveins4} have been covered by Lemma \ref{nolength4cycle}, Proposition \ref{noklein4} and Lemma \ref{A_4notpossible}.
\begin{keyremark}\label{keyremarktotakeaway}
Let $A$ be a simple CM abelian fourfold. Suppose $\gdrbhmath(A)$ is a two-dimensional subtorus of $\mathrm{U}_E$, by the combined efforts of Lemma \ref{nolength4cycle}, Proposition \ref{noklein4} and Lemma \ref{A_4notpossible}, it suffices to consider the scenario where $G'\cong G''\cong \mathrm{D}_4$, i.e., the condition stated in Proposition \ref{noklein4}.   
\end{keyremark}


\subsection{Non-simplicity of Certain CM Abelian Fourfolds}
\label{cmnonsimplicity}
The goal of this section is to handle the scenario mentioned in Key Remark \ref{keyremarktotakeaway} by proving the following lemma. 
\begin{lemma}\label{nosuchsimpleAV}
Let $E/\mathbb{Q}$ be a CM-field of degree $8$, and let $L$ be its Galois
closure in $\qbar$.  Suppose the group homomorphism
\begin{equation}\label{formulanonsimplicity}
 f:\operatorname{Gal}(L/\mathbb Q)\longrightarrow
 \operatorname{Aut}_{\mathbb Z}(X^*(\resgmmath))
\end{equation}
coming from $L/E/\mathbb Q$ has image isomorphic to the dihedral group\[
\mathrm D_4=\langle a,x\mid a^4=x^2=1,\ axa=x\rangle .
\]
Then there does not exist a simple CM abelian fourfold with endomorphism field $E$.
\end{lemma}

We shall use the following standard multiplicity constraint from \cite{shimura1963analytic}.

\begin{lemma}\label{lm2011}
Let $A$ be a simple abelian fourfold of type IV, and let
$K\hookrightarrow \operatorname{End}^{\circ}(A)$ be a quartic CM-field inside its endomorphism algebra.
Then the (multi)set of multiplicities (see Definition \ref{multiplicitydefn}) associated with $\mathrm{Hom}(K,\mathbb{C})$ is $\{2,0,1,1\}$.
\end{lemma}
\begin{proof}
Denote $\mathrm{H}^1(A,\mathbb{Q})$ by $V$. Denote the four embeddings of $K$ into $\mathbb{C}$ by $\sigma,\overline{\sigma},\tau,\overline{\tau}$. Then the eigenspace decomposition induced by $K\xhookrightarrow{}\mathrm{End}(V)$ is $V_{\mathbb{C}}=V_{\sigma}\oplus V_{\overline{\sigma}}\oplus V_{\tau}\oplus V_{\overline{\tau}}$.

By the Albert classification analysis in \cite{mz-4-folds}, the endomorphism algebra of $A$ either is $K$ or is a CM-field of degree 8. Suppose the (multi)set of multiplicities associated with $\mathrm{Hom}(K,\mathbb{C})$ is $\{1,1,1,1\}$. Then by formula \eqref{formulaweilhodgenum}, the Weil structure satisfies $\wedge^2_{K}V \otimes_{\mathbb{Q}} \mathbb{C} \subseteq \mathrm{H}^{1,1}(A,\mathbb{C}).$ Hence every element in the 4-dimensional vector space $\wedge^2_{K}V$ would be a $(1,1)$-Hodge class. First assume $\mathrm{End}^{\circ}(A)$ equals the quartic CM-field $K$, then by the analysis in \cite[561]{mz-4-folds}, $\mathrm{dim}(\mathrm{H}^2(A,\mathbb{Q})\cap\mathrm{H}^{1,1}(A))=\frac{\mathrm{deg}(K/\mathbb{Q})}{2}=2$. Hence this gives a contradiction. Next assume $\mathrm{End}^{\circ}(A)=E$ is a degree 8 CM-field, i.e., $A$ is a simple CM abelian variety. Then because $K\xhookrightarrow{}E$ is a CM-subfield acting with multiplicities $\{1,1,1,1\}$, we can label elements in $\mathrm{Hom}(E,\mathbb{C})$ such that 

\begin{itemize}
    \item The CM type of $A$ is
$
\Phi=\{i_1,i_2,i_3,i_4\},
$
so that
$
\mathrm{Hom}(E,\mathbb{C})
=
\Phi\coprod\overline{\Phi}.
$
    \item The embeddings of $K$ satisfy $\sigma=i_1|_{K}=\overline{i_2}|_{K}; \tau=i_3|_{K}=\overline{i_4}|_{K}$.
\end{itemize} For each element $i\in\mathrm{Hom}(E,\mathbb{C})$, denote by $V_{i}$ the corresponding eigenspace with respect to the $E$-action and also fix a basis $v_{i}$ for each $V_{i}$. The $\mathbb{C}$-linear extension of $(1,1)$-Hodge classes in $\mathrm{H}^2(A,\mathbb{Q})\otimes\mathbb{C}$ is then spanned by $\{v_{i_1}\wedge v_{\overline{i_1}},v_{i_2}\wedge v_{\overline{i_2}},v_{i_3}\wedge v_{\overline{i_3}},v_{i_4}\wedge v_{\overline{i_4}}\}$. The construction of the Weil structures gives $\wedge^2_{K}V \otimes_{\mathbb{Q}} \mathbb{C}=\mathrm{span}_{\mathbb{C}}(\{v_{i_1}\wedge v_{\overline{i_2}},v_{i_2}\wedge v_{\overline{i_1}},v_{i_3}\wedge v_{\overline{i_4}},v_{i_4}\wedge v_{\overline{i_3}}\})$, which clearly has trivial intersection with the complexification of the space of rational Hodge classes. This poses a contradiction.

Next suppose the multiplicities associated with $\mathrm{Hom}(K,\mathbb{C})$ is $\{2,0,2,0\}$. Without loss of generality, we may assume $V_{\sigma}\subset\mathrm{H}^{1,0}$ and $V_{\tau}\subset\mathrm{H}^{1,0}$. Then $V_{\sigmabar}\subset\mathrm{H}^{0,1}$ and $V_{\overline{\tau}}\subset\mathrm{H}^{0,1}$. Consider the subspace $$W=V_{\sigma}\otimes V_{\overline{\sigma}} \oplus V_{\tau}\otimes V_{\overline{\tau}} \subset \wedge^2V\otimes_{\mathbb{Q}}\qbar=\mathrm{H}^2(A,\mathbb{Q})\otimes_{\mathbb{Q}}\qbar.$$ Because $K$ is a CM-field, one can see that $W$ is stable under the action of $\absgalois$, hence descends to an 8-dimensional $\mathbb{Q}$-vector subspace $W'$ of $\wedge^2V=\mathrm{H}^2(A,\mathbb{Q})$. Moreover, every element in $W'$ is in fact a $(1,1)$-Hodge class by the assumption that $V_{\sigma}\subset\mathrm{H}^{1,0}$ and $V_{\tau}\subset\mathrm{H}^{1,0}$. However, if $A$ is a simple CM-abelian variety, the space of $(1,1)$-Hodge classes in $\mathrm{H}^2(A,\mathbb{Q})$ is 4-dimensional and if $\mathrm{End}^{\circ}(A)=K$, it is 2-dimensional. In both scenarios, this poses a contradiction.

\end{proof}

\begin{proof}[Proof of Lemma \ref{nosuchsimpleAV}]
The hypothesis already forces $E/\mathbb Q$ to be Galois. Indeed, formula (\ref{formulanonsimplicity})  is induced by the faithful $\mathrm{Gal}(L/\mathbb{Q})$-action on
$\operatorname{Hom}(E,L)$; hence
$|\operatorname{Gal}(L/\mathbb Q)|=|\mathrm{D}_4|=8=[E:\mathbb Q]$, so $E\cong L$ is a Galois extension.  Thus
$$G:=\operatorname{Gal}(E/\mathbb Q)\cong \mathrm D_4.$$
The only order 2 element in $\mathrm{D}_4$ that commutes with every element is $a^2$ and thus is the complex conjugation. Put
\[
G_1=\{1,x\}, G_2=\{1,ax\};\qquad
K_1=E^{G_1}, K_2=E^{G_2}.
\]
Then $K_1$ and $K_2$ are quartic CM-subfields of $E$.

Assume that a simple CM abelian fourfold $A$ has
$\operatorname{End}^{\circ}(A)=E$.  Fix an embedding
$i_0:E\hookrightarrow L$. We then write $V_g$ for the eigen-subspace of
$\mathrm H^1(A,\mathbb Q)\otimes_{\mathbb Q}L$ attached to
$i_0\circ g$.  Thus
\[
\mathrm H^1(A,\mathbb Q)\otimes L
   =\bigoplus_{g\in G} V_g .
\]
Restriction from $E$ to $K_1$ gives eigenspace decomposition with respect to $K_1\xhookrightarrow{}\mathrm{End}(\betti)$:
\begin{equation}\label{eigenk1}
\mathrm{H}^1(A,\mathbb{Q}) \otimes_{\mathbb{Q}} L=(V_{1} \oplus V_{x}) \oplus(V_{a} \oplus V_{ax}) \oplus (V_{a^2} \oplus V_{a^2 x}) \oplus (V_{a^3} \oplus V_{a^3 x})
\end{equation} 
and restriction from $E$ to $K_2$ gives eigenspace decomposition with respect to $K_2\xhookrightarrow{}\mathrm{End}(\betti)$:
\begin{equation}\label{eigenk2}
\mathrm{H}^1(A,\mathbb{Q}) \otimes_{\mathbb{Q}} L=(V_{1} \oplus V_{ax}) \oplus( V_{a} \oplus V_{a^2 x}) \oplus (V_{a^2} \oplus V_{a^3 x}) \oplus (V_{a^3} \oplus V_{x}).
\end{equation} 

Let $\Phi$ denote one possible CM-type of $A$. Then by definition we have
$\mathrm H^{1,0}(A)=\bigoplus_{g\in\Phi} V_g$ and $\mathrm{Hom}(E,L)=\Phi\coprod\overline{\Phi}$. Since the complex
conjugation is $a^2$, $\Phi$ contains exactly one element from each pair
$\{g,a^2g\}$ and we may assume $\mathrm{id}\in\Phi$. Let $m_{K_1}(\Phi)$ and $m_{K_2}(\Phi)$ denote the multiplicities read
from \eqref{eigenk1} and \eqref{eigenk2} respectively. In Table \ref{cmtypetable}, we list all possible $\Phi$ using the multiplicity constraint from  Lemma \ref{lm2011} applied to
$K_1$.
\begin{table}[ht]
    \centering
    \begin{tabular}{|c|c|c|}
    \hline
    The possible CM-type $\Phi$  & $m_{K_1}(\Phi)$ & $m_{K_2}(\Phi)$ \\
    \hline
    $\{1,x,a,a^3x\}$ & $(2,1,0,1)$ & $(1,1,1,1)$\\
    \hline
    $\{1,x,a^3,ax\}$ & $(2,1,0,1)$ & $(2,0,0,2)$\\
    \hline
    $\{1,a^2x,a,ax\}$ & $(1,2,1,0)$ & $(2,2,0,0)$\\
    \hline
    $\{1,a^2x,a^3,a^3x\}$ &$(1,0,1,2)$& $(1,1,1,1)$\\
    \hline
    \end{tabular}
    \caption{Multiplicities for $K_1$ and $K_2$}
    \label{cmtypetable}
\end{table}
However, according to the table, for every possible CM-type $\Phi$, the multiplicities $m_{K_2}(\Phi)$ associated to $K_2\xhookrightarrow{}\mathrm{End}(\betti)$ is not the multiset
$\{2,0,1,1\}$ required by Lemma \ref{lm2011}. This gives the desired contradiction.
\end{proof}


\subsection{Proof of Theorem \ref{mainthmcm}}
In this section we prove that for a simple CM abelian fourfold $A$, its de Rham-Betti group equals its Mumford-Tate group (see Theorem \ref{mainthmcm}). 
\begin{theorem}[\cite{mz-4-folds}, Theorem 7.6]\label{mzcmthm}
Let $A$ be a simple CM abelian fourfold with complex multiplication by $E$. If there is an imaginary quadratic field $k\subset E$ acting with multiplicities $(2,2)$, then $\mathrm{Hdg}(A)=\mathrm{SU}_{E/k}$. Otherwise, $\mathrm{Hdg}(A)=\mathrm{U}_{E}$.
\end{theorem}

\begin{theorem} \label{mainthmcm}
    For any simple CM abelian fourfold $A$, we have that $\gdrbhmath(A)=\mathrm{Hdg}(A)$. Then by virtue of Theorem \ref{ji_1thmsummarize}, we obtain $\mathrm{MT}(A)=\gdrbmath(A)$.
\end{theorem}

Firstly, Section \ref{longcomp} and \ref{nosuchsimpleAV} already give a lower bound on the dimension of $\gdrbhmath(A)$.
\begin{proposition} \label{biggerthan2}
   For any simple CM abelian fourfold $A$ we have that $\mathrm{dim}(\gdrbhmath(A)) > 2$.
\end{proposition}
\begin{proof}[Proof of Proposition \ref{biggerthan2}]
It is trivial to see that the dimension of $\gdrbhmath(A)$ cannot be 0. Corollary \ref{biggerthan-one} gives $\dim(\gdrbhmath(A))>1$. If
$\dim\gdrbhmath(A)=2$, then Corollary \ref{transitiveins4}, together
with Lemma \ref{nolength4cycle}, Proposition \ref{noklein4}, and Lemma
\ref{A_4notpossible}, force the case described in Key Remark
\ref{keyremarktotakeaway}.  Lemma \ref{nosuchsimpleAV} shows this is impossible for a simple CM abelian fourfold.  Hence the dimension of $\gdrbhmath(A)$ is strictly bigger than $2$.
\end{proof}

\begin{proof}[Proof of Theorem \ref{mainthmcm}]
With the above preparations, we are now ready to apply the argument from 7.6 of \cite{mz-4-folds} verbatim. First suppose the endomorphism field $E$ of $A$ does not contain an imaginary quadratic subfield $k$, then by Lemma \ref{keylemmamz}, there is no subtorus of codimension 1 inside $\mathrm{U}_{E}$. Then $\gdrbhmath(A)=\mathrm{U}_{E}=\mathrm{Hdg}(A)$ because Proposition \ref{biggerthan2} already gives $\mathrm{dim}(\gdrbhmath(A))>2$. 

Now suppose $E$ does contain an imaginary quadratic subfield. First suppose for each imaginary quadratic subfield $k\xhookrightarrow{}E$, the multiplicities associated with $\mathrm{Hom}(k,\mathbb{C})$ is $\{1,3\}$. Combining Lemma \ref{gross} and Theorem \ref{thmchud}, we obtain $\wedge_{k}^4\mathrm{H}^1_{\mathrm{dRB}}(A,\mathbb{Q})\otimes\mathbb{Q}_{\mathrm{dRB}}(2)$ supports no dRB class. Using Lemma \ref{gdrbhinv}, we obtain $\wedge_{k}^4\mathrm{H}^1_{\mathrm{dRB}}(A,\mathbb{Q})^{\gdrbhmath(A)}=\{0\}$. Now assume $\gdrbhmath(A)\subset\mathrm{U}_{E}$ is of dimension 3. Then by Lemma \ref{keylemmamz}, it equals $\mathrm{SU}_{E/k_{0}}$ for some imaginary quadratic field $k_{0}\xhookrightarrow{}E$. Every element in $\wedge_{k_{0}}^4\mathrm{H}^1(A,\mathbb{Q})$ would be fixed by $\mathrm{SU}_{E/k_{0}}=\gdrbhmath(A)$, which is a contradiction. Therefore $\gdrbhmath(A)=\mathrm{U}_{E}=\mathrm{Hdg}(A)$.

Suppose there exists an imaginary quadratic subfield $k_{0}\xhookrightarrow{}E$ acting with multiplicities $\{2,2\}$. Theorem~\ref{mzcmthm} gives
\(\mathrm{Hdg}(A)=\mathrm{SU}_{E/k_0}\). The Hodge group has dimension three.  Since $\gdrbhmath(A)$ has dimension greater than two by Proposition \ref{biggerthan2}, the inclusion
\(\gdrbhmath(A)\subseteq\mathrm{SU}_{E/k_0}\) is an equality. The equality \(\gdrbhmath(A)=\mathrm{Hdg}(A)\) now follows from
Theorem \ref{ji_1thmsummarize}.
\end{proof}

\section{De Rham-Betti Groups of Other Simple Type IV Abelian Fourfolds}\label{chaptercentre}
In this section, we determine the dRB groups of the remaining simple abelian fourfolds of type IV defined over $\qbar$, namely those whose endomorphism algebras are CM-fields of degree 2 or 4. In Subsection \ref{sectioncentre}, we show that the center of the dRB Lie algebra of a type IV abelian fourfold equals the center of its Mumford-Tate Lie algebra. This is achieved by the computation of the dRB groups of certain Weil structures. Thus we can focus on the semisimple part of the dRB Lie algebras. In Subsection \ref{sectiondecomp} and \ref{imaginaryquadraticsection}, we treat the case of an abelian fourfold $A$ with imaginary quadratic endomorphism field and prove Theorem \ref{imaginaryquadraticthm}. As summarized in \ref{outline2} of the Introduction, the main difficulty arises in anti-Weil type abelian fourfolds (see Definition \ref{defnweiltype}). The key input is the weight-one dRB structure constructed in \eqref{antiweilhalfstruct}, inspired by van Geemen's half-twist construction \cite{van2001half}.
In the final subsection, we treat the case of an abelian fourfold $A$ with a quartic CM-endomorphism field and prove Theorem \ref{deg4theorem}. The argument exploits the condition that $A$ is simple and the $\mathbb{Q}$-structure of the dRB group. If $\liegdrbssmath(A)$ were isomorphic to one of the remaining $\mathbb{Q}$-semisimple Lie algebras properly contained in $\mathrm{mt}^{\mathrm{ss}}(A)$, the resulting constraints on any polarization form would be incompatible with the positivity condition required of a
polarization.
 
\subsection{Centres of De Rham-Betti Groups of Simple Type IV Abelian Fourfolds}
\label{sectioncentre}
We start by showing that for a simple type IV abelian fourfold defined over $\qbar$, the centre of its de Rham-Betti Lie algebra equals that of its Mumford-Tate Lie algebra; see Proposition \ref{twocentressame}. The method of proof in this subsection is similar to that of the CM case.
\begin{notation}
For an abelian variety $A$ defined over $\qbar$, write $\liegdrbmath(A)$,
$\mathrm{mt}(A)$, and $\mathrm{hdg}(A)$ for the Lie algebras of its
respective de Rham-Betti, Mumford-Tate, and Hodge groups.
\end{notation}

By the Albert classification analysis in \cite[Section 2.9]{moonen1999hodge}, a non-CM simple type IV abelian fourfold has endomorphism algebra either a quartic CM-field or an imaginary quadratic field.
In the latter case, the multiplicities of the two embeddings of the imaginary quadratic field into $\qbar$ are either
$\{2,2\}$ or $\{1,3\}$ by Proposition 14 of \cite{shimura1963analytic}.

\begin{definition}\label{defnweiltype}
A simple abelian fourfold with imaginary quadratic endomorphism field
$K$ is of \textit{Weil type} if both embeddings of $K$ into $\qbar$ have
multiplicity $2$, and of \textit{anti-Weil type} if the multiplicities
are $\{1,3\}$.
\end{definition}

For a simple type IV abelian variety $A$, let $V=\mathrm H^1(A,\mathbb Q)$, and denote by
$K$ the centre of its Hodge endomorphism algebra. Fix a polarization form $\phi:V\times V\to \mathbb Q(-1)$. Then by \cite[201]{mumford1970abelian}, the Rosati involution associated with $\phi$ on $K$ coincides with the complex conjugation. Moreover, $\phi$ is a non-degenerate bilinear form. This leads to the following well-known fact.
\begin{lemma} \label{phicompatiblewithE}
With the same notation, consider the eigenspace decomposition with respect to the $\mathbb{Q}$-algebra morphism $K\xhookrightarrow{}\mathrm{End}(V)$: $V_{\qbar}=\bigoplus_{\sigma \in \mathrm{Hom}(K,\qbar)}V_{\sigma}$. Then $\phi_{\qbar}: V_{\qbar} \times  V_{\qbar} \rightarrow \qbar(-1)$ satisfies that $\phi_{\qbar}(V_{\sigma},V_{\sigma'})=0$ if $\sigma' \neq \overline{\sigma}$. Moreover, for each $\sigma\in \mathrm{Hom}(K,\qbar)$, we can find a $\qbar$-basis $\{e_{i}\}_{i\in I}$ for $V_{\sigma}$ and a $\qbar$-basis $\{f_{j}\}_{j\in J}$ for $V_{\overline{\sigma}}$ such that $\phi_{\qbar}(e_{i},f_{j})=0$ if $i \neq j$ and $\phi_{\qbar}(e_{i},f_{j})=1$ if $i=j$.
\end{lemma}

The Hodge Lie algebra of a simple type IV abelian fourfold is summarized as follows. 
\begin{theorem}[\cite{mz-4-folds}, Table 1 in p.578]
\label{deg2mz4}
Let $A$ be a simple type IV abelian fourfold and denote $\mathrm{End}^{\circ}(A)$ by $K$ and $\betti$ by $V$. Fix a polarization form
$\phi:V\times V\to \mathbb Q(-1)$. 
\begin{itemize}
    \item If $\mathrm{deg}(K/\mathbb{Q})=4$ or if $A$ is of
    anti-Weil type, then $\mathrm{hdg}(A)$ is equal to
    \[
    \mathrm{u}_{K}(V,\phi):=
    \{M \in \mathrm{End}_{K}(V)\mid
    \phi(Mv,w)+\phi(v,Mw)=0;\ \forall v,w \in V\}.
    \]
    The centre of $\mathrm{hdg}(A)$ is the subspace
    $\{x \in K\mid x+\overline{x}=0\}$.
    
    \item If $A$ is a Weil type abelian fourfold, then
    $\mathrm{hdg}(A)=\{M \in \mathrm{u}_{K}(V,\phi)\mid
    \mathrm{tr}_{K}(M)=0\}$. The centre of $\mathrm{hdg}(A)$ in this
    case is trivial.
\end{itemize}

\end{theorem}

To prove Proposition \ref{twocentressame}, we will use results about de Rham-Betti groups of certain Weil structures, which are computed in Lemma \ref{13group} and Lemma \ref{deg4weil}.
\begin{lemma} \label{13group}
Let $A$ be a simple anti-Weil type abelian fourfold defined over $\qbar$ with imaginary
quadratic endomorphism field $K$. Then the de Rham-Betti group of the Weil structure $W_{\mathrm{Weil}}:=\wedge_{K}^4\mathrm{H}^1_{\mathrm{dRB}}(A,\mathbb{Q})$ is $\mathrm{Res}_{K/\mathbb{Q}}\gmmath$.
\end{lemma}

\begin{proof}
According to Lemma \ref{gross} the comparison matrix of $W_{\mathrm{Weil}} \otimes_{\mathbb{Q}} \qbar$ is diagonal with diagonal entries $b_{K}^{1} (2\pi \sqrt{-1}/b_{K})^{3}$ and $b_{K}^{3} (2\pi \sqrt{-1}/b_{K})^{1}$. By Theorem \ref{thmchud} and the fact that the dimension of the dRB group
is at least the transcendence degree of the periods (see \cite[Corollary 2.11]{bost2016some}), we obtain $\mathrm{dim}(\gdrbmath(W_{\mathrm{Weil}})) \geq 2$. On the other hand, there is a natural morphism of $\mathbb{Q}$-algebras $K \xhookrightarrow{} \mathrm{End}_{\gdrbmath(W_{\mathrm{Weil}})} (W_{\mathrm{Weil}})$: on each summand $\wedge^4 V_{\sigma}$ of $W_{\mathrm{Weil}}\otimes\qbar$, an element $k\in K$ acts as scalar multiplication by $\sigma(k)$. Clearly this action is Galois equivariant and moreover preserves the dRB structure on $W_{\mathrm{Weil}}$. This realizes $\mathrm{Res}_{K/\mathbb{Q}}\gmmath$ as a maximal torus inside $\mathrm{GL}(W_{\mathrm{Weil}})$. Therefore we have $\gdrbmath(W_{\mathrm{Weil}}) \xhookrightarrow{} \mathrm{Res}_{K/\mathbb{Q}}\gmmath$. As $\dim(\mathrm{Res}_{K/\mathbb{Q}}\gmmath)=2$, the equality follows.
\end{proof}
Now let $A$ be an abelian fourfold defined over $\qbar$ whose endomorphism field $K$ is a degree $4$ CM-field. We consider
\begin{equation}\label{deg4drbweildefn}
   W_{\mathrm{Weil}}^{\mathrm{dRB}}
   :=\wedge_K^2\mathrm{H}^1_{\mathrm{dRB}}(A,\mathbb{Q})
\end{equation}
as a dRB substructure of
$\wedge_{\mathbb{Q}}^2\mathrm{H}^1_{\mathrm{dRB}}(A,\mathbb{Q})$.
The computation of its dRB group is analogous to the CM
case, once the issue of weights is handled carefully. We start by showing the existence of a weight-labeling object in the Tannakian category $\langle W_{\mathrm{Weil}}^{\mathrm{dRB}} \rangle^{\otimes}$.

\begin{lemma} \label{evendegreeweildrb}The dRB structure $\mathbb{Q}_{\mathrm{dRB}}(-2)$ is an object in the Tannakian category $\langle W_{\mathrm{Weil}}^{\mathrm{dRB}} \rangle^{\otimes}$.
\end{lemma}
\begin{proof}
The polarization form $\phi$ induces a symmetric bilinear form $\phi': \mathrm{H}^2(A,\mathbb{Q}) \times \mathrm{H}^2(A,\mathbb{Q})  \rightarrow \mathbb{Q}(-2)$. On basic tensors $v_1\wedge v_2\in\mathrm{H}^2(A,\mathbb{Q})$ and $v_3\wedge v_4\in\mathrm{H}^2(A,\mathbb{Q})$, this is defined by $$\phi'(v_1\wedge v_2,v_3\wedge v_4)=\phi(v_1,v_3)\phi(v_2,v_4)-\phi(v_1,v_4)\phi(v_2,v_3).$$ Since $\phi$ is a morphism of dRB structures, so is $\phi'$. Its restriction to the dRB substructure $W_{\mathrm{Weil}}^{\mathrm{dRB}} \xhookrightarrow{} \mathrm{H}^2_{\mathrm{dRB}}(A,\mathbb{Q})$ is nonzero: after
extending scalars to $\qbar$, choose basis
$\{e_1,e_2\}$ for $V_{\sigma_0}$ and
$\{f_1,f_2\}$ for $V_{\overline{\sigma_0}}$ as in Lemma
\ref{phicompatiblewithE}; then $\phi'_{\qbar}(e_1\wedge e_2,f_1\wedge f_2)=1$.
Thus
\begin{equation}\label{weilweilsurj}
\phi'|_{W_{\mathrm{Weil}}^{\mathrm{dRB}}}:
W_{\mathrm{Weil}}^{\mathrm{dRB}}\otimes
W_{\mathrm{Weil}}^{\mathrm{dRB}}\twoheadrightarrow
\mathbb Q_{\mathrm{dRB}}(-2)
\end{equation}
realizes $\mathbb Q_{\mathrm{dRB}}(-2)$ as a quotient in the Tannakian
category generated by $W_{\mathrm{Weil}}^{\mathrm{dRB}}$.
\end{proof}

\begin{lemma} \label{deg4weil}
The de Rham-Betti group of the Weil structure $W_{\mathrm{Weil}}^{\mathrm{dRB}}$ is isogenous to $\gmmath \times \mathrm{U}_{K}$.
\end{lemma}

\begin{proof}
The proof strategy is the same as Theorem \ref{mainthmcm}. Denote the Weil Hodge structure $W_{\mathrm{Weil}}^{\mathrm{Hodge}}=\wedge_{K}^2\mathrm{H}^1(A,\mathbb{Q})$ by $\mathcal{W}_1$ and write $\mathcal{W}_2:=W_{\mathrm{Weil}}^{\mathrm{dRB}}$. We have the following commutative diagram. 
$$\begin{tikzcd}
\gdrbmath(\mathcal{W}_2) \arrow[r, hook] \arrow[d, hook]    & \mathrm{MT}(\mathcal{W}_1) \arrow[ld, hook] \\
{\mathrm{GL}(\wedge_{K}^2\mathrm{H}^1(A,\mathbb{Q}))} &                 \end{tikzcd}$$ 
Similar to the proof of Lemma \ref{13group}, we obtain a morphism of $\mathbb{Q}$-algebras $K \xhookrightarrow{} \mathrm{End}_{\mathrm{Hdg}}(\mathcal{W}_1)$. Since $\mathrm{deg}(K/\mathbb{Q})=\mathrm{dim}(\wedge_{K}^2\mathrm{H}^1(A,\mathbb{Q}))$, we obtain $\mathrm{Hdg}(\mathcal{W}_1) \subset \mathrm {U}_{K}$. Denote the natural isogeny $\gmmath\times\mathrm{Hdg}(\mathcal{W}_1)\rightarrow\mathrm{MT}(\mathcal{W}_1)$ by $\Psi_1$.

By Tannakian formalism, $\gdrbmath(\mathcal{W}_2)\subset\mathrm{GL}(\wedge_{K}^2\mathrm{H}^1(A,\mathbb{Q}))$ is the natural restriction of the action of $\gdrbmath(A)$ on $\wedge^2\mathrm{H}^1(A,\mathbb{Q})$. Therefore the group of homotheties $\gmmath\xhookrightarrow{}\gdrbmath(A)$ acts on $\wedge_{K}^2\mathrm{H}^1(A,\mathbb{Q})$ via the scalar multiplication by $t^{-2}$. Since $\gdrbmath(\mathcal{W}_2)$ is an algebraic group, similar to \cite[Theorem 4.2]{ji_1}, we obtain $\gdrbmath(\mathcal{W}_2)$ contains the group of homotheties in $\mathrm{GL}(\wedge_{K}^2\mathrm{H}^1(A,\mathbb{Q}))$.

We set $$\gdrbhmath(\mathcal{W}_2):=\mathrm{pr}_2(\Psi_{1}^{-1}(\gdrbmath(\mathcal{W}_2))^{\circ})$$ where $\mathrm{pr}_2:\gmmath \times \mathrm{Hdg}(A)\rightarrow \mathrm{Hdg}(A)$ is the projection map. By \cite[Theorem 2.6]{ksv2026rhambetticlassescoefficients}, $\gdrbmath(\mathcal{W}_2)$ is a connected algebraic group. This gives the following commutative diagram \begin{equation}\label{weilgdrbhdiagram}\begin{tikzcd}
\gmmath \times \gdrbhmath(\mathcal{W}_2) \arrow[r, "\Psi_2",  two heads] \arrow[d, hook] & \gdrbmath(\mathcal{W}_2) \arrow[d, hook] \\
\gmmath \times \mathrm{Hdg}(\mathcal{W}_1) \arrow[r, "\Psi_1", two heads]               & \mathrm{MT}(\mathcal{W}_1)              
\end{tikzcd}\end{equation} We now show 
$\gdrbhmath(\mathcal W_2)=\mathrm U_E$.  If it had dimension
$<2$, Lemma \ref{keylemmamz} would give either
$\gdrbhmath(\mathcal W_2)=\{1\}$ or
$\gdrbhmath(\mathcal W_2)=\mathrm{SU}_{K/k}$ for some imaginary
quadratic subfield $k\subset K$.

If $\gdrbhmath(\mathcal W_2)=\{1\}$, then every element of
$\mathcal W_2\otimes\mathbb Q_{\mathrm{dRB}}(1)$ is a dRB class.  By
Theorem \ref{bostoriginal}, dRB classes in $\mathrm{H}^2$ of abelian varieties are Hodge classes;
however, the space of Hodge classes
in $\mathrm H^2(A,\mathbb Q)\otimes\mathbb Q(1)$ has
dimension $2$, whereas $\mathcal{W}_2$ has dimension 4, a contradiction.

Suppose $\gdrbhmath(\mathcal W_2)=\mathrm{SU}_{K/k}$. Consider the following two Weil dRB structures 
\begin{equation}\label{twoweildrb}
\wedge_k^2\mathcal W_2\cong\wedge_k^4\mathrm H^1_{\mathrm{dRB}}(A,\mathbb Q)
\end{equation} where the isomorphism of the two Weil dRB structures follows from Definition \ref{weilstructure}.  Each element of \eqref{twoweildrb} would then be fixed by $\gdrbhmath(\mathcal W_2)$. By formula \eqref{weilweilsurj}, the homothety subgroup of $\gdrbmath(\mathcal W_2)$ acts on $\mathbb{Q}_{\mathrm{dRB}}(2)$ via scalar multiplication by $t^{2}$ and $\gdrbhmath(\mathcal W_2)$ acts trivially on $\mathbb{Q}_{\mathrm{dRB}}(2)$ by \cite[Lemma 3.5]{ji_1}. Hence all
elements in the twisted dRB structure
$\wedge_k^2\mathcal W_2\otimes\mathbb Q_{\mathrm{dRB}}(2)
\cong\wedge_k^4\mathrm H^1_{\mathrm{dRB}}(A,\mathbb Q)
\otimes\mathbb Q_{\mathrm{dRB}}(2)$ would be dRB classes.  But Lemma \ref{lm2011} gives multiplicities $\{2,0,1,1\}$ for the $K$-action and since $k$ is a CM-subfield, the induced $k$-action has multiplicities $\{1,3\}$. Lemma \ref{gross}, together
with Theorem \ref{thmchud}, shows this twisted Weil
structure does not support any nonzero dRB classes. Thus
$\gdrbhmath(\mathcal W_2)=\mathrm U_{K}$ and the proof concludes by invoking the isogeny $\Psi_2$ in \eqref{weilgdrbhdiagram}.

\end{proof}

\begin{proposition} \label{twocentressame}
    For any simple type IV abelian fourfold $A$ defined over $\qbar$, the center of its de Rham-Betti Lie algebra is equal to the center of its Mumford-Tate Lie algebra. 
\end{proposition}

\begin{proof}
We already have
$\mathrm{Z}(\gdrbmath(A))\hookrightarrow \mathrm{Z}(\mathrm{MT}(A))$ by \cite[Lemma 2.7]{ji_1}. Hence it suffices to compare dimensions. If $A$ is of Weil type, Theorem \ref{deg2mz4} gives
$\mathrm{Z}(\mathrm{MT}(A))^{\circ}=\gmmath$, while
\cite[Theorem 4.2]{ji_1} gives
$\gmmath\subset \mathrm{Z}(\gdrbmath(A))$; hence the centres of the corresponding Lie algebras agree. In the other two scenarios, by the Tannakian functoriality and the fact that $\gdrbmath(A)$ is reductive, we have 
\[
\mathrm{Z}(\gdrbmath(A))\twoheadrightarrow
\gdrbmath(W_{\mathrm{Weil}});
\]
see also \cite[Lemma 2.9]{ji_1}. If $A$ is of anti-Weil type, Lemma \ref{13group} implies that $\dim \mathrm{Z}(\gdrbmath(A))\geq 2$ and Theorem \ref{deg2mz4} gives
$\dim \mathrm{Z}(\mathrm{MT}(A))=2$, so $\mathrm{Z}(\liegdrbmath(A))=\mathrm{Z}(\mathrm{mt}(A))$.
If $\mathrm{deg}(K/\mathbb{Q})=4$, Lemma \ref{deg4weil} implies
$\dim \mathrm{Z}(\gdrbmath(A))\geq 3$, while Theorem \ref{deg2mz4} gives
$\dim \mathrm{Z}(\mathrm{MT}(A))=3$. Thus the equality follows.

\end{proof}

\subsection{Decomposition of the de Rham-Betti Representations}
\label{sectiondecomp}

Let $A$ be a simple abelian variety of type IV defined over $\qbar$ with
$\operatorname{End}^{\circ}(A)=K$, where $K$ is a CM-field of degree
$d$. Let $V=\mathrm H^1(A,\mathbb Q)$.  Since the image of the canonical
representation
$
\eta:\liegdrbmath(A)\longrightarrow \mathrm{gl}(V)
$
commutes with $K$, each eigenspace in the decomposition $V_{\qbar}=\bigoplus_{\sigma\in\operatorname{Hom}(K,\qbar)}V_\sigma$ for the action of $K$ is preserved by $\eta_{\qbar}$. 
For each $\sigma$, let $\eta_\sigma$ denote the subrepresentation of $\liegdrbmath(A)_{\qbar}$ on $V_\sigma$. Recall that $\liegdrbmath(A)$ is a reductive Lie algebra; see \cite[Corollary 5.11]{ksv2026rhambetticlassescoefficients}.
We then set
\[
\rho_\sigma:=\eta_\sigma^{\mathrm{ss}}:
\liegdrbssmath(A)_{\qbar}\longrightarrow \mathrm{gl}(V_\sigma),
\qquad
\liegdrbssmath(A):=[\liegdrbmath(A),\liegdrbmath(A)] .
\]

\begin{lemma}\label{lemmaeigenspaceirr}
Each $\rho_\sigma$ is an irreducible
representation. Moreover, a polarization induces an isomorphism of representations between $\rho_\sigma^{*}$ and $\rho_{\overline{\sigma}}$.  
\end{lemma}
\begin{proof}Suppose $V_{\sigma}$ as a $\liegdrbssmath(A)_{\qbar}$-representation decomposes into $V_1 \oplus V_2$. Denote the subspace $\bigoplus_{\sigma_{i}\in\mathrm{Hom}(K,\qbar),\sigma_{i}\neq\sigma} V_{\sigma_{i}}$ by $V'$. Then elements in $S:=\{\alpha_1|_{{V_1}}\oplus\alpha_2|_{{V_2}}\oplus\alpha_3|_{V'}:\alpha_i\in\qbar\}
   \subset \mathrm{End}(V_{\qbar})$
commute with $\liegdrbssmath(A)_{\qbar}$, where $\alpha|_{V}$ denotes scalar multiplication by $\alpha$ on $V$.  Moreover, \cite[Lemma 2.7]{ji_1} gives that $\mathrm{Z}(\liegdrbmath(A))_{\qbar}\subset K\otimes\qbar$, which acts as scalar multiplication on each eigenspace. Therefore elements in $S$ commute with $\liegdrbmath(A)_{\qbar}=\mathrm{Z}(\liegdrbmath(A))_{\qbar}\oplus\liegdrbssmath(A)_{\qbar}$. But $S$ does not lie in $K\otimes\qbar=\mathrm{End}_{\liegdrbmath(A)}(V)\otimes\qbar$, which gives the contradiction. For the second statement, fix a polarization form $\phi$ on $A$. By Lemma \ref{phicompatiblewithE}, 
$\phi_{\qbar}$ restricts to a non-degenerate pairing between conjugate eigenspaces $\phi_{\qbar}:V_\sigma\times V_{\bar\sigma}
       \longrightarrow \qbar(-1).$
Since $\liegdrbssmath(A)$ preserves the polarization form infinitesimally, for
$l\in\liegdrbssmath(A)_{\qbar}$, $v\in V_\sigma$, and
$w\in V_{\bar\sigma}$,
we have $
\phi_{\qbar}(\rho_\sigma(l)v,w)
+
\phi_{\qbar}(v,\rho_{\bar\sigma}(l)w)=0.
$
Therefore $\phi_{\qbar}$
identifies $V_\sigma$ equivariantly with $V_{\bar\sigma}^{*}$.
\end{proof}

Using the above analysis and the invariant theoretic methods to determine the Mumford-Tate group of $A$ from \cite[Section 7.4, p.~572]{mz-4-folds}, we obtain the following classification result; see also \cite[Proposition 5.2.1]{ji2026thesis} of the author's thesis.

\begin{lemma}
\label{imaginaryquadraticclassification}
Let $A$ be a simple abelian fourfold with imaginary quadratic
endomorphism field $K$. We keep the same notation as above. Then precisely one of the following possibilities occurs for the faithful irreducible representation
\[
\rho_\sigma:\liegdrbssmath(A)_{\qbar}
\longrightarrow \mathrm{gl}(V_\sigma).
\]
\begin{enumerate}
    \item $\liegdrbssmath(A)_{\qbar}\cong\mathrm{sl}(2)$ and
    $V_\sigma\cong V(3)$;
    \item $\liegdrbssmath(A)_{\qbar}\cong
    \mathrm{sl}(2)\times\mathrm{sl}(2)$ and
    $V_\sigma\cong V(1)\boxtimes V(1)$;
    \item $\liegdrbssmath(A)_{\qbar}\cong\mathrm{sp}(4)$ and
    $V_\sigma$ is its standard representation;
    \item $\liegdrbssmath(A)_{\qbar}\cong\mathrm{sl}(4)$ and
    $V_\sigma$ is its standard representation.
\end{enumerate}
\end{lemma}

Recall from Theorem \ref{deg2mz4}, that any abelian fourfold $A$ with imaginary quadratic endomorphism field has $\mathrm{mt}^{\mathrm{ss}}(A)_{\qbar}\cong\mathrm{sl}(4)$, which corresponds to the fourth case of Lemma \ref{imaginaryquadraticclassification}.  The following elementary representation-theoretic computation distinguishes this case from the other three.

\begin{lemma}\label{wedgecandidateclassification}
For each of the first three representations listed in Lemma \ref{imaginaryquadraticclassification}, the representation $\wedge^2\rho_\sigma$ is the direct sum of two non-isomorphic irreducible representations, whereas it is irreducible in the fourth case. More precisely, \begin{align*}
\wedge^2V(3)
  &\cong V(4)\oplus V(0),\\
\wedge^2\bigl(V(1)\boxtimes V(1)\bigr)
  &\cong\bigl(V(2)\boxtimes V(0)\bigr)
       \oplus\bigl(V(0)\boxtimes V(2)\bigr),\\
\wedge^2\mathrm{Std}_{\mathrm{sp}(4)}
  &\cong V(\varpi_2)\oplus V(0),\\
\wedge^2\mathrm{Std}_{\mathrm{sl}(4)}
  &\cong V(\varpi_2),
\end{align*}
Here $V(0)$ denotes the trivial representation.
\end{lemma}

\begin{corollary}\label{weiltypedrb}
Let $A$ be a simple abelian fourfold of Weil type. Then
$\gdrbmath(A)=\mathrm{MT}(A)$.
\end{corollary}

\begin{proof}
The method is identical to the one determining the Mumford-Tate groups of such abelian varieties in \cite[Section 7.4, p.~572]{mz-4-folds}. We briefly recall the key point. By Proposition \ref{twocentressame}, we have $\mathrm{Z}(\gdrbmath(A))=\mathrm{Z}(\mathrm{MT}(A))=\gmmath$. Therefore we obtain the equality of Lie algebras $\liegdrbssmath(A)=\liegdrbhmath(A)$. We suppress the Tate twists in the rest of the proof. Because the Picard rank of $A$ is $\frac{\mathrm{deg}(K/\mathbb{Q})}{2}=1$, the $\qbar$-span of Hodge classes in $\mathrm{H}^2(A,\mathbb{Q})$ necessarily lies in the summand $V_{\sigma}\otimes V_{\overline{\sigma}}\cong V_{\sigma}\otimes V_{\sigma}^{*}\subset \mathrm{H}^2(A,\mathbb{Q})\otimes\qbar$. However, Theorem \ref{bostoriginal} states that Hodge and de Rham-Betti classes coincide in degree two cohomology groups. By Lemma \ref{wedgecandidateclassification}, if $\liegdrbssmath(A)_{\qbar}=\liegdrbhmath(A)_{\qbar}$ were isomorphic to $\mathrm{sl}(2)$ or $\mathrm{sp}(4)$, the summand $\wedge^2 V_{\sigma}$ contains nonzero $\liegdrbhmath(A)_{\qbar}$-invariant classes, which is a contradiction. If $\liegdrbhmath(A)_{\qbar}$ were isomorphic to $\mathrm{sl}(2)\times\mathrm{sl}(2)$, then $V_{\sigma}\otimes V_{\sigma}\cong V_{\sigma}\otimes V_{\overline{\sigma}}^{*}$ contains nonzero $\liegdrbhmath(A)_{\qbar}$-invariant classes. However, the $\qbar$-span of dRB/Hodge classes in $\mathrm{End}(\betti)\otimes\qbar$ lies in the summand $V_{\sigma}\otimes V_{\sigma}^{*}\oplus V_{\overline{\sigma}}\otimes V_{\overline{\sigma}}^{*}$, which again gives the contradiction. Therefore $\liegdrbssmath(A)_{\qbar}\cong\mathrm{sl}(4)$. Combining this with the equality of centres gives
$\liegdrbmath(A)=\mathrm{mt}(A)$. Both groups are connected, hence
$\gdrbmath(A)=\mathrm{MT}(A)$.
\end{proof}
\begin{remark}\label{13remark}
If $A$ is of anti-Weil type with endomorphism field $K$, then Proposition \ref{twocentressame} combined with Theorem \ref{deg2mz4} show that $\mathrm{Z}(\gdrbhmath(A))$ is isogenous to $\mathrm{U}_{K}$. A direct computation then shows that the degree-two invariants constructed in the proof of the preceding corollary for the first three possibilities of $\rho_{\sigma}$ listed in Lemma \ref{imaginaryquadraticclassification} are not fixed by $\mathrm{Z}(\gdrbhmath(A))$. Hence the argument from the previous Corollary \ref{weiltypedrb} cannot be used to determine the de Rham-Betti group of $A$. The Mumford-Tate group of $A$, however, is determined by Ribet in \cite[Theorem 3]{ribet1983hodge}. His proof uses the Deligne torus definition of Mumford-Tate groups in a fundamental way, and thus cannot be transferred to the de Rham-Betti setting.
\end{remark}

\subsection{Anti-Weil Type Abelian Fourfolds}\label{imaginaryquadraticsection}

The goal of this subsection is to prove Theorem \ref{imaginaryquadraticthm}. The case of Weil type abelian fourfolds has already been treated in Corollary \ref{weiltypedrb}. The main point is therefore to treat the case of anti-Weil type abelian fourfolds, as explained in Remark \ref{13remark}.

\begin{theorem}\label{imaginaryquadraticthm}
Let $A$ be a simple abelian fourfold defined over $\qbar$ such that $\mathrm{End}^{\circ}(A)=K$, where $K$ is an imaginary quadratic field. Then $\gdrbmath(A)=\mathrm{MT}(A)$.
\end{theorem}

We assume from now on that $A$ is of anti-Weil type (See Definition \ref{defnweiltype}) with endomorphism field $K$. Set
$V=\betti$ and write
$\mathrm{Hom}(K,\qbar)=\{\sigma,\overline\sigma\}$. We keep the same notation as Subsection \ref{sectiondecomp}. We may assume the Hodge decomposition satisfies
$\dim V_\sigma^{1,0}=1, \dim V_{\overline\sigma}^{1,0}=3$. Our goal is to exclude the first three possibilities for $\rho_{\sigma}$ from Lemma \ref{imaginaryquadraticclassification} using the decomposition result from Lemma \ref{wedgecandidateclassification}. However, the relevant decomposition of the representation $\wedge^2 \rho_{\sigma}$ arises at degree-two cohomology of $A$, and the resulting discrepancy between invariants can only be detected at $\mathrm{End}(\mathrm{H}^2_{\mathrm{dRB}}(A,\mathbb{Q}))$. This exceeds the scope of the Analytic Subgroup Theorem used in \cite{bost2016some} or \cite[Section 7.5.3]{andre2004introduction}. To get around this obstacle, we will use a construction first made by van Geemen in \cite[Proposition 2.8]{van2001half}. First, let $E_K$ be an elliptic curve defined over $\qbar$ with complex multiplication by $K$ and with the following CM-type: Consider the eigenspace decomposition with respect to the $K$-action
\begin{equation}\label{elliptichomologydecomp}
\mathrm H_1(E_K,\mathbb Q)\otimes_{\mathbb Q}\qbar
=E_\sigma\oplus E_{\overline\sigma},
\end{equation}
then the CM-type of $E_{K}$ satisfies $E_\sigma$ is of Hodge type $(-1,0)$ and
$E_{\overline{\sigma}}$ is of Hodge type $(0,-1)$. We also denote
\[
\mathrm H_{1,\mathrm{dRB}}(E_K,\mathbb Q)
:=\mathrm H^1_{\mathrm{dRB}}(E_K,\mathbb Q)^{\vee}
\]
for its corresponding dual dRB structure.

Consider the following $\qbar$-subspace
\begin{equation}\label{antiweilhalfstruct}
U_{1/2,\qbar}:=
\bigl(\wedge^2V_\sigma\otimes E_{\overline\sigma}\bigr)
\oplus
\bigl(\wedge^2V_{\overline\sigma}\otimes E_\sigma\bigr)
\subset
\mathrm H^2(A,\mathbb Q)\otimes
\mathrm H_1(E_K,\mathbb Q)\otimes\qbar.
\end{equation}
Note that $U_{1/2,\qbar}$ is preserved by the $\absgalois$-action. Hence it descends to a $\mathbb{Q}$-subspace $U_{1/2}$ of $\mathrm H^2(A,\mathbb Q)\otimes
\mathrm H_1(E_K,\mathbb Q)$.
\begin{lemma}\label{antiweilpolarizable}
The $\mathbb{Q}$-vector space $U_{1/2}$ is a polarizable weight-one Hodge
substructure of $\mathrm H^2(A,\mathbb Q)\otimes
\mathrm H_1(E_K,\mathbb Q)$, i.e., it only has Hodge types $(1,0)$ and $(0,1)$. Moreover, $U_{1/2}$ can also be viewed as a de Rham-Betti substructure of $\mathrm H^2_{\mathrm{dRB}}(A,\mathbb Q)\otimes
\mathrm H_{1,\mathrm{dRB}}(E_K,\mathbb Q)$.
\end{lemma}

\begin{proof}
This is already explained in Proposition 2.8 of \cite{van2001half}. We write down the Hodge decomposition explicitly.
\begin{equation}
\begin{aligned}
\mathrm{H}^{1,0}(U_{1/2,\mathbb{C}})
    &= V_{\sigma}^{1,0}\otimes V_{\sigma}^{0,1}
       \otimes E_{\overline{\sigma},\mathbb{C}}
       \oplus \wedge^2 V_{\overline{\sigma}}^{1,0}
       \otimes E_{\sigma,\mathbb{C}}, \\
\mathrm{H}^{0,1}(U_{1/2,\mathbb{C}})
    &=V_{\overline{\sigma}}^{1,0}\otimes V_{\overline{\sigma}}^{0,1}
       \otimes E_{\sigma,\mathbb{C}}
       \oplus \wedge^2 V_{\sigma}^{0,1}
       \otimes E_{\overline{\sigma},\mathbb{C}}.
\end{aligned}
\end{equation}
Moreover, both $\betti$ and $\mathrm{H}_1(E_{K},\mathbb{Q})$ admit a polarization and therefore the weight-one sub-Hodge structure $U_{1/2}$ is polarizable. By Theorem \ref{ji_1thmsummarize}, if $U_{1/2}$ is preserved by $\mathrm{MT}(A\times E_{K})$, it is also preserved by $\gdrbmath(A\times E_{K})$. Thus $U_{1/2}$ is also a natural dRB substructure.
\end{proof}
Relating to $\mathrm H^2(A,\mathbb Q)\otimes
\mathrm H_1(E_K,\mathbb Q)$, the structure $U_{1/2}$ satisfies stronger properties, as demonstrated in the following lemma.

\begin{lemma}\label{antiweilprojector}
There exists an idempotent endomorphism $p_U:\mathrm H^2(A,\mathbb Q)
\otimes\mathrm H_{1}(E_K,\mathbb Q)
\longrightarrow
\mathrm H^2(A,\mathbb Q)
\otimes\mathrm H_{1}(E_K,\mathbb Q)$ whose image is precisely
$U_{1/2}$. Furthermore, it is induced by a $\mathbb Q$-coefficient algebraic cycle defined over $\qbar$.
\end{lemma}
\begin{proof}
    Let $K=\mathbb{Q}(\sqrt{-d})$ where $d$ is a positive integer. Denote by $J_1$ the image of $\sqrt{-d}$ in $\mathrm{End}(\betti)$ and $J_2$ the image of $\sqrt{-d}$ in $\mathrm{End}(\mathrm{H}_1(E_{K},\mathbb{Q}))$. Denote by $S$ the linear operator on $\wedge^2\betti$ which sends $x\wedge y$ to $J_1\circ x\wedge y+x\wedge J_1\circ y$. Then checking the eigenvalues of these linear operators, under the map $\frac{1}{2}(1+\frac{S\otimes J_2}{2d})$, only the subspace $\wedge^2V_{\sigma}\otimes E_{\overline\sigma}\oplus\wedge^2V_{\overline{\sigma}}\otimes E_{\sigma}\oplus V_{\sigma}\otimes V_{\overline{\sigma}}\otimes E_{\sigma}\oplus V_{\sigma}\otimes V_{\overline{\sigma}}\otimes E_{\overline{\sigma}}$ remains. Denote by $R$ the linear operator on $\wedge^2 \betti$ which sends $x\wedge y$ to $J_1\circ x\wedge J_1\circ y$, then under the idempotent projection map $\frac{1}{2}(1-\frac{R}{d})$, only the summand $\wedge^2V_{\sigma}\oplus \wedge^2V_{\overline{\sigma}}$ remain. Hence the algebraic endomorphism $\frac{1}{2}(1-\frac{R\otimes 1}{d})\circ\frac{1}{2}(1+\frac{S\otimes J_2}{2d})$ is the desired idempotent.
\end{proof}

By Lemma \ref{antiweilpolarizable}, there exists a complex abelian sixfold $B$ such that $\mathrm{H}^1(B,\mathbb{Q})$ is isomorphic to $U_{1/2}$ as Hodge structures. Moreover, the following is true.
\begin{proposition}\label{antiweilabelianrealization}
There exists an abelian sixfold $B$ defined over $\qbar$ and an
isomorphism of dRB structures
\begin{equation}\label{antiweildrbisomorphism}
\mathrm H^1_{\mathrm{dRB}}(B,\mathbb Q)
\xrightarrow{\ \sim\ }U_{1/2}.
\end{equation}
\end{proposition}
We need two results to establish the above proposition. The first one is a renowned theorem by Deligne.
\begin{theorem}[\cite{deligne1982hodgecycles}, Theorem 2.11]
\label{deligneabsolutehodge}
Every Hodge class on a complex abelian variety is an absolute Hodge
class.
\end{theorem}

\begin{lemma}[\cite{patrikis2016anabelian}, Lemma 3.6]
\label{pvzdescentlemma}
Let $k\subset\mathbb C$ be a countable field, and let $\overline{k}$
be its algebraic closure in $\mathbb C$. Let $B$ be a complex abelian
variety such that $B^\tau$ is isogenous to $B$ for every
$\tau\in\mathrm{Aut}(\mathbb C/k)$. Then there exists an abelian
variety $B_0$ over $\overline{k}$ such that
$B_0\otimes_{\overline{k}}\mathbb C\cong B$.
\end{lemma}

\begin{proof}[Proof of Proposition \ref{antiweilabelianrealization}]
Let $B$ be a complex abelian sixfold such that $\mathrm{H}^1(B,\mathbb{Q})$ is isomorphic to $U_{1/2}$ as Hodge structures via
\[
\iota:\mathrm H^1(B,\mathbb Q)\xrightarrow{\ \sim\ }U_{1/2}.
\]
This induces a morphism of Hodge structures
\begin{equation}\label{antiweilhodgeembedding}
F:\mathrm H^1(B,\mathbb Q)\longrightarrow
\mathrm H^2(A,\mathbb Q)\otimes\mathrm H_1(E_K,\mathbb Q)
\end{equation}
whose image is $U_{1/2}$. In particular,
\begin{equation}\label{antiweilprojectorrelation}
p_U\circ F=F.
\end{equation}

The morphism $F$ is represented by a Hodge class
\[
Z_{F}\in\mathrm H^4(B\times A^{\vee}\times E_K^{\vee},\mathbb Q)(2).
\]
Theorem \ref{deligneabsolutehodge} implies that this class is absolute Hodge. 

Let $\tau\in\mathrm{Aut}(\mathbb C/\qbar)$. Note $A$ and $E_K$ are defined over $\qbar$. The conjugate class $Z_{F}^\tau$ is then a Hodge class in $\mathrm H^4(B^{\tau}\times A^{\vee}\times E_K^{\vee},\mathbb Q)(2)$.
This induces a morphism of Hodge structures
\[
F^\tau:\mathrm H^1(B^\tau,\mathbb Q)\longrightarrow
\mathrm H^2(A,\mathbb Q)\otimes\mathrm H_1(E_K,\mathbb Q).
\]
By the explicit construction in Lemma \ref{antiweilprojector}, $p_U$ is in fact defined over $\qbar$. We thus have $p_U^{\tau}=p_U$. Conjugating \eqref{antiweilprojectorrelation} gives
\[
p_U\circ F^{\tau}=p_U^{\tau}\circ F^\tau=F^\tau.
\]
Consequently
$\mathrm{Im}(F^\tau)\subseteq\mathrm{Im}(p_U)$ lies in $U_{1/2}$. Conjugation
does not change the rank of the map $F$, and both spaces have dimension $12$;
hence $\mathrm{Im}(F^\tau)=U_{1/2}$. Thus $\mathrm H^1(B^\tau,\mathbb Q)\cong U_{1/2}\cong
\mathrm H^1(B,\mathbb Q)$ as rational Hodge structures. Hence $B^\tau$ is isogenous to $B$. Applying Lemma \ref{pvzdescentlemma} with $k=\qbar$, we obtain a model of $B$ over $\qbar$, which we also denote by $B$. Since $Z_{F}$ is already a Hodge class on a product of abelian varieties, it is a de Rham-Betti class, because of the inclusion of the dRB group inside the Mumford-Tate group, as explained in Subsection \ref{introsection1}. Hence
\eqref{antiweilhodgeembedding} is a morphism of dRB structures. Its Betti component identifies $\mathrm H^1(B,\mathbb Q)$ with the image of
$p_U$, which induces the dRB isomorphism
\eqref{antiweildrbisomorphism}.
\end{proof}

Via the abelian sixfold $B$, Proposition \ref{antiweilabelianrealization} tells us that every dRB endomorphism of $U_{1/2}$ is automatically a Hodge endomorphism. 

\begin{corollary}\label{antiweilendomorphisms}
Under the isomorphism \eqref{antiweildrbisomorphism}, we have
\[
\mathrm{End}_{\mathrm{dRB}}(U_{1/2})
=\mathrm{End}_{\mathrm{Hodge}}(U_{1/2}).
\]
\end{corollary}

\begin{proof}
This follows from Theorem \ref{bostoriginal} and the proposition above.
\end{proof}

\begin{proposition}\label{antiweilexclusion}
For a simple anti-Weil type abelian fourfold $A$, the first three cases
of Lemma \ref{imaginaryquadraticclassification} do not occur.
\end{proposition}

\begin{proof}
Suppose one of the first three cases occurs. By Lemma
\ref{wedgecandidateclassification}, there is a decomposition into
non-isomorphic irreducible representations of $\gdrbmath(A)_{\qbar}$
\[
\wedge^2V_\sigma=W_1\oplus W_2.
\]
The conjugate eigenspace has the corresponding decomposition
$\wedge^2V_{\overline\sigma}=W'_1\oplus W'_2$.

Note that $\gdrbmath(A\times E_K)\subset \gdrbmath(A)\times \gdrbmath(E_K)$. Hence the projection map
$$\mathrm{pr}_1: U_{1/2,\qbar}=
(W_1\otimes E_{\overline\sigma})
\oplus(W_2\otimes E_{\overline\sigma})
\oplus(W'_1\otimes E_\sigma)
\oplus(W'_2\otimes E_\sigma) \rightarrow W_1\otimes E_{\overline\sigma}
$$
is
$\gdrbmath(A\times E_K)_{\qbar}$-equivariant. 
By Corollary \ref{antiweilendomorphisms}, the same projection map is equivariant under $\mathrm{MT}(A\times E_K)_{\qbar}$, which acts via $\mathrm{MT}(A)_{\qbar}$ on $W_1$ and acts on $E_{\overline\sigma}$ via a character of $\mathrm{MT}(E_K)_{\qbar}$. Moreover, the canonical morphism from $\mathrm{MT}(A\times E_K)_{\qbar}$ to $\mathrm{MT}(A)_{\qbar}$ is surjective. This implies that
$W_1\subset\wedge^2V_\sigma$ is stable under
$\mathrm{MT}(A)_{\qbar}$ and hence stable under $\mathrm{mt}^{\mathrm{ss}}(A)_{\qbar}$.

On the other hand, Theorem \ref{deg2mz4} gives
$\mathrm{mt}^{\mathrm{ss}}(A)_{\qbar}
=\mathrm{sl}(V_\sigma)\cong \mathrm{sl}_4$, and $\wedge^2V_\sigma$ is an irreducible $\mathrm{sl}_4$-representation. This contradicts the existence of the nonzero proper invariant subspace $W_1$.
\end{proof}

\begin{proof}[Proof of Theorem \ref{imaginaryquadraticthm}]
If $A$ is of Weil type, the result is Corollary \ref{weiltypedrb}.
Suppose that $A$ is of anti-Weil type. Proposition
\ref{antiweilexclusion} excludes the first three cases of Lemma
\ref{imaginaryquadraticclassification}. Therefore
$\liegdrbssmath(A)_{\qbar}
=\mathrm{sl}(V_\sigma)
=\mathrm{mt}^{\mathrm{ss}}(A)_{\qbar}$. Proposition \ref{twocentressame} identifies the centres of
$\liegdrbmath(A)$ and $\mathrm{mt}(A)$. Hence
$\liegdrbmath(A)=\mathrm{mt}(A)$. Since both algebraic groups are connected, we conclude that $\gdrbmath(A)=\mathrm{MT}(A)$.
\end{proof}

\subsection{Abelian Fourfolds with Quartic CM Endomorphism Fields}\label{deg4section}
The goal of this section is to prove the following theorem.
\begin{theorem} \label{deg4theorem}
    Let $A$ be a simple abelian fourfold defined over $\qbar$ such that $\mathrm{End}^{\circ}(A)= K$ where $K$ is a quartic CM-field. Then $\gdrbmath(A)=\mathrm{MT}(A)$.
\end{theorem}
\begin{remark}
Our method of proving Theorem \ref{deg4theorem} is different from the method of computing the Mumford-Tate group of such an abelian fourfold from Section 7.5 of \cite{mz-4-folds}. The authors of \cite{mz-4-folds} use the definition of the Mumford-Tate group as the smallest $\mathbb{Q}$-algebraic subgroup of $\mathrm{GL}(V)$ whose set of $\mathbb{R}$-points contains the image of the Deligne torus in an essential way. However, in the current state of art, de Rham-Betti groups are only defined in the Tannakian formalism.  
\end{remark}
\begin{remark}
By Proposition \ref{twocentressame}, we have $\mathrm{Z}(\gdrbhmath(A))=\mathrm{Z}(\mathrm{Hdg}(A))$ which is isogenous to $\mathrm{U}_{K}$. Consequently, as in the case of anti-Weil type abelian fourfold (see Remark \ref{13remark}), a straightforward invariant-theoretic computation at degree-two cohomology groups does not suffice. Moreover, it is not clear to the author whether the method of using van Geemen's half-twist construction in the proof of Theorem \ref{imaginaryquadraticthm} extends directly in this setting. Hence we adopt a different approach to proving Theorem \ref{deg4theorem}.    
\end{remark}
Proposition \ref{twocentressame} already gives $\mathrm{Z}(\liegdrbmath(A))=\mathrm{Z}(\mathrm{mt}(A))$. Hence it suffices to study the semisimple part $\liegdrbssmath(A)$. We use notations from Section \ref{sectiondecomp}. We set 
$\mathrm{Hom}(K,\qbar)=\{\sigma,\sigmabar,\tau,\overline{\tau}\}$ and denote $\betti$ by $V$. Then $V_{\qbar}=V_{\sigma}\oplus V_{\overline{\sigma}}\oplus V_{\tau}\oplus V_{\overline{\tau}}$ is the eigenspace decomposition with respect to the action of $K$ on $V$.

\begin{lemma}\label{deg4repclassification}
Using the above notation, the images of $\rho_{\sigma}:\liegdrbssmath(A)_{\qbar} \rightarrow \mathrm{gl}(V_{\sigma})$ and $\rho_{\tau}: \liegdrbssmath(A)_{\qbar} \rightarrow \mathrm{gl}(V_{\tau})$ are $\mathrm{sl}(V_{\sigma})$ and $\mathrm{sl}(V_{\tau})$. Moreover the representation $\liegdrbssmath(A)_{\qbar} \rightarrow \mathrm{gl}(V_{\qbar})$ has only one of the following two possibilities.
  \begin{enumerate}
      \item\label{case1} $\liegdrbssmath(A)_{\qbar}\cong\mathrm{sl}(2)$ and its image under $\rho_{\sigma} \oplus \rho_{\tau}$ in $\mathrm{sl}(V_{\sigma}) \times \mathrm{sl}(V_{\tau})$ is the graph of a Lie algebra isomorphism between $\mathrm{sl}(V_{\sigma})$ and $\mathrm{sl}(V_{\tau})$.
      \item $\liegdrbssmath(A)_{\qbar}\cong\mathrm{sl}(2) \times \mathrm{sl}(2)$ and its image under $\rho_{\sigma} \oplus \rho_{\tau}$ is the entire $\mathrm{sl}(V_{\sigma}) \times \mathrm{sl}(V_{\tau})$.
  \end{enumerate}
  
\end{lemma}
\begin{proof}
The same argument from \cite[575]{mz-4-folds} can be applied; we fill in more details for reader's convenience. The images of the semisimple Lie algebra $\liegdrbssmath(A)_{\qbar}$ under $\rho_{\sigma}$ and $\rho_{\tau}$ are semisimple Lie subalgebras of $\mathrm{sl}(V_{\sigma})$ and $\mathrm{sl}(V_{\tau})$; see \cite[Corollary 5.3]{humphreys2012introduction}. But $\mathrm{dim}(V_{\sigma})= \mathrm{dim}(V_{\tau})=2$, hence their images are either $\{0\}$ or the entire $\mathrm{sl}(2)$. By Lemma \ref{lemmaeigenspaceirr}, $\rho_{\sigma}$ and $\rho_{\tau}$ are irreducible representations. Hence their images are the entire $\mathrm{sl}(V_{\sigma})$ and $\mathrm{sl}(V_{\tau})$.

Denoting $\rho_{\sigma} \oplus \rho_{\tau}$ by $\iota$, we have the following commutative diagram
    $$\begin{tikzcd}
                                                                                                                                          & \mathrm{sl}(V_{\tau})   &                                                                                                               \\
\liegdrbssmath(A)_{\qbar} \arrow[rr,"\iota"] \arrow[ru, "\rho_{\tau}=p_2\circ \iota", two heads] \arrow[rd, "\rho_{\sigma}=p_1\circ \iota"', two heads] &                         & \mathrm{sl}(V_{\sigma})\times\mathrm{sl}(V_{\tau}) \arrow[ld, "p_1", two heads] \arrow[lu, "p_2"', two heads] \\
                                                                                                                                          & \mathrm{sl}(V_{\sigma}) &                                                                                                              
\end{tikzcd}$$
By the Lie algebra version of Goursat's Lemma \cite[Section 3.1]{moonen1999hodge} and the fact that $\mathrm{sl}(2)$ is a simple Lie algebra, we have that $\iota(\liegdrbssmath(A)_{\qbar})$ is either $\mathrm{sl}(V_{\sigma}) \times \mathrm{sl}(V_{\tau})$ or the graph of a Lie algebra isomorphism between $\mathrm{sl}(V_{\sigma})$ and $\mathrm{sl}(V_{\tau})$. However, Lemma \ref{lemmaeigenspaceirr} gives the isomorphism of representations $\rho_{\overline{\sigma}}\cong \rho_{\sigma}^{*}$ and  $\rho_{\overline{\tau}}\cong \rho_{\tau}^{*}$. Therefore the image of $\liegdrbssmath(A)_{\qbar}$ in $\mathrm{sl}(V_{\qbar})$ is completely determined by its restriction to $\mathrm{sl}(V_{\sigma}) \times \mathrm{sl}(V_{\tau})$. Hence $\liegdrbssmath(A)$ is either isomorphic to a $\qbar/\mathbb{Q}$ form of $\mathrm{sl}(2)$ or isomorphic to a $\qbar/\mathbb{Q}$ form of $\mathrm{sl}(2) \times \mathrm{sl}(2)$.
\end{proof}

We now exclude Case (\ref{case1}) of Lemma \ref{deg4repclassification} where $\liegdrbssmath(A)$ is a $\qbar/\mathbb{Q}$-form of $\mathrm{sl}(2)$. Such forms are classified by the quaternion algebras over $\mathbb{Q}$. Recall a quaternion algebra $\mathbb{Q}(a,\lambda)$ has basis $\{1,i,j,k\}$ whose ring structure is given by $i^2=a,\: j^2=\lambda,\: ij=-ji=k$ where $a,\lambda\in\mathbb{Q}^{\times}$. We denote by $\mathbb{Q}(a,\lambda)^{\circ}$ the $\mathbb{Q}$-vector subspace spanned by $\{i,j,k\}$. We can also put a Lie bracket structure on $\mathbb{Q}(a,\lambda)$ by defining $[x,y]=xy-yx$, making $\mathbb{Q}(a,\lambda)^{\circ}$ a Lie subalgebra.
\begin{lemma}[\cite{serre1979galois}, III.1.4]\label{sl2qformlemma}
Every $\qbar/\mathbb{Q}$-form of $\mathrm{sl}(2)$ is isomorphic to the Lie algebra $\mathbb{Q}(a,\lambda)^{\circ}$ for some $a,\lambda \in \mathbb{Q}^{\times}$. 
\end{lemma}
\begin{proof}
For the proof, see Section III.1 of \cite{serre1979galois}. Given a quaternion algebra $\mathbb{Q}(a,\lambda)$, we fix an explicit isomorphism of Lie algebras $\mathbb{Q}(a,\lambda)^{\circ} \otimes \qbar \cong \mathrm{sl}(2)$ as follows
\begin{equation}\label{sl2explicitiso}
  h=i\otimes\frac{1}{\sqrt{a}};\: x=\frac{1}{2\lambda}(j+k\otimes\frac{1}{\sqrt{a}});\: y=\frac{1}{2}(j-k\otimes\frac{1}{\sqrt{a}})
\end{equation} where we fix a square root of $a$ in $\qbar$ and denote it by $\sqrt{a}$. Note that $h,x,y$ indeed define a $\mathrm{sl}(2)$-triple in the sense of \cite[Chapter 7]{humphreys2012introduction}: $[x,y]=h, [h,x]=2x, [h,y]=-2y$.
\end{proof}
The complex conjugation fixes $\sqrt{a}$ if $a>0$, and sends 
$\sqrt{a}$ to $-\sqrt{a}$ if $a<0$, which yields:

\begin{corollary} \label{relations between y and x}
With the explicit basis fixed in isomorphism (\ref{sl2explicitiso}), one has $\overline{h}=-h$; $\lambda\overline{x}=y$ if $a<0$ and $\overline{h}=h$; $\overline{x}=x$; $\overline{y}=y$ if $a>0$.
\end{corollary}
Assume Case (\ref{case1}) from Lemma \ref{deg4repclassification}. We then have a morphism of $\mathbb{Q}$-Lie algebras $\eta: \liegdrbssmath(A)\cong\mathbb{Q}(a,\lambda)^{\circ} \rightarrow \mathrm{gl}(V)$ such that its $\qbar$-linear extension decomposes into a direct sum of standard representations of $\mathrm{sl}(2)$
$$\eta_{\qbar}:\mathrm{sl}(2)\rightarrow \mathrm{gl}(V_{\sigma})\oplus \mathrm{gl}(V_{\overline{\sigma}})\oplus \mathrm{gl}(V_{\tau})\oplus \mathrm{gl}(V_{\overline{\tau}}).$$ Because $K$ is a CM-field, we also have that $\overline{V_{\sigma}}=V_{\overline{\sigma}}$ and $\overline{V_{\tau}}=V_{\overline{\tau}}$. 
\begin{lemma} \label{conjugate and weight vector V(1)}
 With the above setup and isomorphism (\ref{sl2explicitiso}), let $v_{m} \in V_{\iota}$ ($\iota \in \{\sigma,\sigmabar,\tau,\overline{\tau}\}$) be an eigenvector of $\rho_{\iota}(h)$ of weight (eigenvalue) $m$ ($m \in \{-1,1\}$). If $a<0$, then $\overline{v_{m}} \in V_{\overline{\iota}}$ remains an eigenvector of $\rho_{\overline{\iota}}(h)$ but has weight $-m$. If $a>0$, then $\overline{v_{m}} \in V_{\overline{\iota}}$ remains an eigenvector of $\rho_{\overline{\iota}}(h)$ with the same weight $m$.
\end{lemma}
\begin{proof}
 By isomorphism (\ref{sl2explicitiso}), an eigenvector $v_{m} \in V_{\iota}$ of weight $m$ satisfies $$\rho_{\iota}(h) \circ v_{m}=\eta_{\qbar}(h) \circ v_{m}=\eta(i)\otimes \frac{1}{\sqrt{a}} \circ v_{m}=mv_{m}.$$  
 Now suppose $a<0$. Note that $\eta(i)$ has entries in $\mathbb{Q}$ with respect to a $\mathbb{Q}$-basis of $V$, therefore applying the complex conjugation to the above equation, $\eta(i)$ remains unchanged. Thus we have
\begin{equation*}
   \overline{\rho_{\iota}(h) \circ v_{m}}=\eta(i) \otimes \overline{\frac{1}{\sqrt{a}}} \circ \overline{v_{m}}=-\eta(i) \otimes \frac{1}{\sqrt{a}} \circ \overline{v_{m}}=-\rho_{\overline{\iota}}(h) \circ \overline{v_{m}}=m\overline{v_{m}}
\end{equation*}
Hence $\overline{v_{m}}$ remains an eigenvector of $h$ but has weight $-m$.
The case where $a>0$ follows similarly.
\end{proof}

Finally, observe that any polarization form on $V$ is preserved by the image of $\liegdrbssmath(A)$ infinitesimally. Therefore the assumption that $\liegdrbssmath(A)$ were isomorphic to $\mathbb{Q}(a,\lambda)^{\circ}$ would actually impose linear relations on values of the polarization form. This is made precise in Lemma \ref{deg4Polarisation properties k merged}.

\begin{lemma} \label{deg4Polarisation properties k merged} 
Suppose there exists an abelian fourfold $A$ defined over $\qbar$ with $\mathrm{End}^{\circ}(A)= K$ where $K$ is a quartic CM-field, and furthermore assume the hypothetical scenario where $\liegdrbssmath(A) \cong \mathbb{Q}(a,\lambda)^{\circ}$ with $a,\lambda \neq 0$. Fix the $\mathrm{sl}(2)$-triple $(h,x,y)$ for $\mathbb{Q}(a,\lambda)^{\circ}\otimes\qbar \cong \mathrm{sl}(2)$ as in isomorphism (\ref{sl2explicitiso}). Then for any polarization form $\phi:V \times V \rightarrow \mathbb{Q}$, we can find a basis $\{v_{-1},v_1\}$ for $V_{\sigma}$, $\{w_{-1},w_1\}$ for $V_{\overline{\sigma}}$, $\{v'_{-1},v'_{1}\}$ for $V_{\tau}$ and $\{w'_{-1},w'_{1}\}$ for $V_{\overline{\tau}}$ such that the following conditions are satisfied:
\begin{enumerate}
    \item\label{eigenvaluecond} Each $v_{m}$, $w_{m}$, $v'_{m}$ and $w'_{m}$ is of weight $m$.
    \item\label{eigenvectorconjugate} If $a < 0$ then $w_{m}=\overline{v_{-m}}$ and $w'_{m}=\overline{v'_{-m}}$; if $a > 0$ then $w_{m}=\overline{v_{m}}$ and $w'_{m}=\overline{v'_{m}}$.
   
    \item\label{polarizationeigenvector1} Each $V_{\sigma}$, $V_{\overline{\sigma}}$, $V_{\tau}$ and $V_{\overline{\tau}}$ is totally isotropic for $\phi_{\qbar}$, and \begin{equation}\label{sameweight0}\phi_{\qbar}(v_{m},w_{m})=\phi_{\qbar}(v'_{m},w'_{m})=0\qquad m\in\{-1,1\}\end{equation}
    \item\label{polarizationeigenvector2} The remaining pairings satisfy the following relations depending on the sign of $a$:
    \begin{itemize}
        \item If $a < 0$: $\phi_{\qbar}(v_{-1},w_1)+\lambda\phi_{\qbar}(v_1,w_{-1})=0; \phi_{\qbar}(v'_{-1},w'_1)+\lambda\phi_{\qbar}(v'_1,w'_{-1})=0$
       
        \item If $a > 0$: $\phi_{\qbar}(v_{-1},w_1)+\phi_{\qbar}(v_1,w_{-1})=0;
        \phi_{\qbar}(v'_{-1},w'_1)+\phi_{\qbar}(v'_1,w'_{-1})=0$
       \end{itemize}
\end{enumerate} 
\end{lemma}

\begin{proof}
We fix a weight 1 eigenvector $v_1$ of $\rho_{\sigma}(h)$ in $V_{\sigma}$, and fix $v'_1$ a weight 1 eigenvector of $\rho_{\tau}(h)$ in $V_{\tau}$. The remaining basis elements of $V_{\sigma}$ and $V_{\tau}$ are
$$v_{-1}=\rho_{\sigma}(y)\circ v_1\in V_{\sigma}, \quad v'_{-1}=\rho_{\tau}(y)\circ v'_1\in V_{\tau}.$$
To define the basis elements in $V_{\overline{\sigma}}$ and $V_{\overline{\tau}}$, we split into two cases based on the sign of $a$:
\begin{itemize}
    \item If $a < 0$, let 
    $w_{-1}=\overline{v_1} \in V_{\overline{\sigma}},w_1=\overline{v_{-1}} \in V_{\overline{\sigma}}; w'_{-1}=\overline{v'_1}\in V_{\overline{\tau}},w'_1=\overline{v'_{-1}} \in V_{\overline{\tau}}$

    \item If $a > 0$, let $
    w_{-1}=\overline{v_{-1}}\in V_{\overline{\sigma}},w_1=\overline{v_{1}} \in V_{\overline{\sigma}}; w'_{-1}=\overline{v'_{-1}} \in V_{\overline{\tau}},w'_{1}=\overline{v'_1}\in V_{\overline{\tau}}$
\end{itemize}
By Lemma \ref{conjugate and weight vector V(1)}, Condition (\ref{eigenvaluecond}) and (\ref{eigenvectorconjugate}) in the lemma are satisfied by construction. 

For the third condition, because any polarization form is preserved infinitesimally by the image of $\liegdrbssmath(A)_{\qbar}$, for any $v,w\in V_{\qbar}$ and $l\in\liegdrbssmath(A)_{\qbar}$ we have
$$\phi_{\qbar}(l\circ v,w)+\phi_{\qbar}(v,l\circ w)=0$$ Letting $l=h$ and swapping in $v\in V_{\iota}$ and $w\in V_{\overline{\iota}}$, we obtain Condition (\ref{polarizationeigenvector1}). Moreover, Lemma \ref{phicompatiblewithE} already implies that each eigenspace is totally isotropic with respect to $\phi_{\qbar}$. 

For the final Condition (\ref{polarizationeigenvector2}), we let $l=y$ and obtain that
\begin{equation}\label{fourthrelation}
\phi_{\qbar}(y \circ v_1,w_1)+\phi_{\qbar}(v_1,y\circ w_1)=0;\quad \phi_{\qbar}(y \circ v'_1,w'_1)+\phi_{\qbar}(v'_1,y\circ w'_1)=0.
\end{equation}
If $a < 0$, we apply the relation $\lambda\overline{x}=y$ from Corollary \ref{relations between y and x} to obtain the last set of relations of the lemma. If $a > 0$, we instead use the fact that $\overline{y}=y$ to obtain them.
\end{proof}

\begin{proposition}\label{deg4nosuchabelian4fold}
There does not exist an abelian fourfold $A$ with $\mathrm{End}^{\circ}(A)= K$ where $K$ is a quartic CM-field such that $\liegdrbssmath(A)_{\qbar} \cong \mathrm{sl}(2)$. Thus Case (\ref{case1}) of Lemma \ref{deg4repclassification} does not occur.
\end{proposition}

By Lemma \ref{lm2011}, the multiplicities associated with 
$\mathrm{Hom}(K,\mathbb{C})$ are $\{2,0,1,1\}$. Without loss of generality, 
assume $m_{\sigma}=2$, $m_{\overline{\sigma}}=0$, and 
$m_{\tau}=m_{\overline{\tau}}=1$. For simplicity of notation, we still write $V_{\iota}$ for $V_{\iota}\otimes_{\qbar}\mathbb{C}$ for each $\iota\in\{\sigma,\overline{\sigma},\tau,\overline{\tau}\}$. Then Definition \ref{multiplicitydefn} yields the decompositions:
\begin{equation}\label{deg4hodgedecomp}
    V^{1,0} = V_{\sigma}^{1,0}\oplus V_{\overline{\sigma}}^{1,0}\oplus 
    V_{\tau}^{1,0} \oplus V_{\overline{\tau}}^{1,0}; \qquad   V^{0,1} = V_{\sigma}^{0,1}\oplus V_{\overline{\sigma}}^{0,1}\oplus 
    V_{\tau}^{0,1} \oplus V_{\overline{\tau}}^{0,1}
\end{equation}
where $\mathrm{dim}(V_{\sigma}^{1,0}) = \mathrm{dim}(V_{\overline{\sigma}}^{0,1}) = 2$, 
$\mathrm{dim}(V_{\overline{\sigma}}^{1,0}) = \mathrm{dim}(V_{\sigma}^{0,1}) = 0$, 
and $V_{\tau}^{1,0}, V_{\tau}^{0,1}, V_{\overline{\tau}}^{1,0}, V_{\overline{\tau}}^{0,1}$ all have dimension 1. Furthermore, Lemma \ref{phicompatiblewithE} implies that 
(\ref{deg4hodgedecomp}) satisfies the orthogonality condition $V_{\overline{\tau}}^{1,0}=(V_{\tau}^{1,0})^{\perp}\cap V_{\overline{\tau}}$. 

Recall the polarization form of a polarizable Hodge structure satisfies the positivity condition. We follow the convention of \cite[Definition 3.1.6]{huybrechts2016lectures}: 
\begin{lemma}[Positivity Condition for Polarization Form]\label{positivitylemma}
Let $(V,\phi)$ be a polarized weight one Hodge structure, then for any $v\in V^{1,0}-\{0\}$ we have
\begin{equation}\label{deg4pos}
\phi_{\mathbb{R}}(v+\overline{v},\sqrt{-1}v-\sqrt{-1}\overline{v})=-2\sqrt{-1}\phi_{\mathbb{C}}(v,\overline{v})>0
\end{equation}
\end{lemma}

\begin{proof}[Proof of Proposition \ref{deg4nosuchabelian4fold}]
 
Assume such $A$ in the statement of the proposition exists. First suppose that $\liegdrbssmath(A)\cong\mathbb{Q}(a,\lambda)^{\circ}$ with $a<0$. Then there exists a weight 1 Hodge structure on $V$ satisfying decomposition (\ref{deg4hodgedecomp}), polarized by a bilinear form $\phi$ whose values satisfy relations from Lemma \ref{deg4Polarisation properties k merged}. By Lemma \ref{positivitylemma}, for any $v \in V^{1,0}-\{0\}$, the inequality (\ref{deg4pos}) needs to hold. Fix a basis for $V_{\sigma}$, $V_{\overline{\sigma}}$, $V_{\tau}$ and $V_{\overline{\tau}}$ as in Lemma \ref{deg4Polarisation properties k merged}. Recall that we have assumed in decomposition (\ref{deg4hodgedecomp}) that $\mathrm{dim}(V_{\sigma}^{1,0})=2$. Then for any $v_{\sigma}^{1,0}=x_{-1}v_{-1}+x_{1}v_{1}\in V_{\sigma}^{1,0}-\{0\}=V_{\sigma}-\{0\}$, the positivity condition (\ref{deg4pos}) of the polarization form is equivalent to
\begin{equation*}
-2\sqrt{-1}\phi_{\mathbb{C}}(x_{-1}v_{-1}+x_{1}v_{1},\overline{x_{-1}}w_1+\overline{x_1}w_{-1})>0; \qquad \forall (x_{-1},x_{1})\in\mathbb{C}^2-\{0\}.
\end{equation*} 
 Denote $\sqrt{-1}\phi_{\mathbb{C}}(v_1,w_{-1})$ by $M$, using Lemma \ref{deg4Polarisation properties k merged}, the above inequality simplifies to
\begin{equation}
    -2M(x_1\overline{x_{1}}-\lambda x_{-1}\overline{x_{-1}})>0; \qquad \forall (x_{-1},x_{1})\in\mathbb{C}^2-\{0\}.
\end{equation}  Thus we deduce that $\lambda \in \mathbb{Q}_{<0}$.

Formula (\ref{deg4pos}) also needs to hold for elements in $V_{\tau}^{1,0}$ and $V_{\overline{\tau}}^{1,0}$. By decomposition (\ref{deg4hodgedecomp}), $\mathrm{dim}(V_{\tau}^{1,0})=1$. We now choose a $\mathbb{C}$-basis $v_{\tau}^{1,0}=y_{-1}v'_{-1}+y_{1}v'_{1}\in V_{\tau}-\{0\}$ for $V_{\tau}^{1,0}$ where $(y_{-1},y_1)\in\mathbb{C}^2-\{0\}$.
Then formula (\ref{deg4pos}) for $v_{\tau}^{1,0}$ is equivalent to
\begin{equation*}
-2\sqrt{-1}\phi_{\mathbb{C}}(y_{-1}v'_{-1}+y_{1}v'_{1},\overline{y_{-1}}w'_1+\overline{y_1}w'_{-1})>0.
\end{equation*}
Denoting $\sqrt{-1}\phi_{\mathbb{C}}(v'_1,w'_{-1})$ by $N$ and using relations from Lemma \ref{deg4Polarisation properties k merged}, this simplifies to \begin{equation}\label{deg4vtau10}
    -2N(y_1\overline{y_{1}}-\lambda y_{-1}\overline{y_{-1}})>0.
\end{equation} 
Moreover, for any element $y'_{-1}w'_{-1}+y'_{1}w'_{1}\in V_{\overline{\tau}}^{1,0}-\{0\}$, formula (\ref{deg4pos}) is equivalent to

\begin{equation}\label{vbartau10}
2N(-\lambda y'_1\overline{y'_{1}}+y'_{-1}\overline{y'_{-1}})>0.
\end{equation}
Since $\lambda<0$, inequalities (\ref{deg4vtau10}) and (\ref{vbartau10}) give a clear contradiction.

Next assume $\liegdrbssmath(A)\cong\mathbb{Q}(a,\lambda)^{\circ}$ with $a>0$. Then for any non-zero $v_{\sigma}^{1,0}=x_{-1}v_{-1}+x_{1}v_{1}\in V_{\sigma}^{1,0}=V_{\sigma}$ with $(x_{-1},x_{1})\in\mathbb{C}^2-\{0\}$, formula (\ref{deg4pos}) is equivalent to
\begin{equation*}
-2\sqrt{-1}\phi_{\mathbb{C}}(x_{-1}v_{-1}+x_{1}v_{1},\overline{x_{-1}}w_{-1}+\overline{x_1}w_{1})>0.
\end{equation*}
Using the relations from Lemma \ref{deg4Polarisation properties k merged}, and denoting $\sqrt{-1}\phi_{\mathbb{C}}(v_1,w_{-1})$ by $M'$, this simplifies to
\begin{equation}
    -2M'(x_{1}\overline{x_{-1}}-x_{-1}\overline{x_{1}})>0;\qquad \forall (x_{-1},x_{1})\in \mathbb{C}^2-\{0\}.
\end{equation} But when $(x_{-1},x_1)\in\mathbb{R}^2-\{0\}$, the above expression is equal to zero, a contradiction.

\end{proof}

\begin{proof}[Proof of Theorem \ref{deg4theorem}]
 Proposition \ref{twocentressame} gives $\mathrm{Z}(\mathrm{mt}(A))=\mathrm{Z}(\liegdrbmath(A))$. Now Proposition \ref{deg4nosuchabelian4fold} has excluded Case (\ref{case1}) of Lemma \ref{deg4repclassification}. Therefore we can conclude that $\liegdrbssmath(A)_{\qbar}\cong\mathrm{sl}(2)\times\mathrm{sl}(2)$. Thus by (7.5) of \cite{mz-4-folds}, one can see that $\liegdrbssmath(A)_{\qbar}=\mathrm{mt}(A)^{\mathrm{ss}}_{\qbar}$. Recall that $\liegdrbmath(A) \subset \mathrm{mt}(A)$ as a $\mathbb{Q}$-Lie subalgebra. Therefore we obtain that $\liegdrbmath(A)=\mathrm{mt}(A)$. Because $\mathrm{MT}(A)$ is a connected algebraic group, we can conclude that $\gdrbmath(A)=\mathrm{MT}(A)$.
\end{proof}
\renewcommand*{\bibfont}{\footnotesize}
\printbibliography[
heading=bibintoc,
title={References}
]

\end{document}